\documentclass[12pt,reqno]{amsart}
\usepackage{}
\usepackage{lmodern}
\usepackage{mathrsfs}
\usepackage{tensor}
\usepackage[all]{xy}
\usepackage{amsfonts, amssymb,amsmath, amsthm, amsxtra, latexsym, amscd}
\usepackage[latin1]{inputenc}
\usepackage{xcolor}
\usepackage{graphicx}
\usepackage{bm}

\theoremstyle{definition}
\newtheorem*{thm1}{Theorem 1}
\newtheorem*{thm2}{Theorem 2}

\newtheorem{thm}{Theorem}[section]
\newtheorem{cnj}[thm]{Conjecture}

\newtheorem{cor}[thm]{Corollary}
\newtheorem{lem}[thm]{Lemma}
\newtheorem{prop}[thm]{Proposition}
\newtheorem{exm}[thm]{Example}
\newtheorem{rem}[thm]{Remark}

\numberwithin{equation}{section}

\makeatletter
\@namedef{subjclassname@2020}{\textup{2020} Mathematics Subject Classification}
\makeatother

\bigskip

\def\refn#1.#2{\expandafter\def\csname#1\endcsname{[#2]}}
\def\refnr#1.{\csname#1\endcsname}

\begin{document}

\baselineskip  1.2pc
\title[invariance of essential normality]{ The biholomorphic  invariance of essential normality on  bounded symmetric domains}
\author[L. Ding]{Lijia Ding}
\address{School of Mathematics and Statistics,
Zhengzhou University,
Zhengzhou,
Henan 450001,
P. R. China}
\email{ljding@zzu.edu.cn}

\maketitle
\begin{abstract}

This paper mainly concerns the biholomorphic invariance of $p$-essential normality of  Hilbert modules on bounded symmetric domains. By establishing new integral formulas concerning rational function kernels for the Taylor functional calculus, we prove a biholomorphic invariance result related to the $p$-essential normality.  Furthermore, for quotient analytic Hilbert submodules determined by analytic varieties, we develop an algebraic approach to proving that the $p$-essential normality is preserved invariant if the coordinate multipliers are replaced by arbitrary automorphism multipliers. Moreover, the Taylor spectrum of the compression tuple is calculated under a mild condition, which gives a solvability result of the corona problem for quotient submodules. As applications, we extend the recent results on the equivalence between  $\infty$-essential normality and hyperrigidity.

\end{abstract}
\maketitle

\section{Introduction}

The study of $p$-essentially normal Hilbert modules originates from Arveson's seminal works  \cite{Arv98, Arv00, Arv05} in which it is applied to investigate the dilation theory and the geometric invariant theory of the commuting operator tuple. Recently, the $p$-essential normality of Hilbert modules on some bounded pseudoconvex domains has continuously attracted more attention,  for example in the papers \cite{Arv07, DTY16, EnE15,  GWZ23, GuW20, KS12, KS16}.  The results reveal the profound interplay among complex algebraic geometry, index theory, operator algebra and operator theory.   

The purpose of this paper is to establish the biholomorphic invariance of $p$-essential normality of  Hilbert modules 
on bounded symmetric domains, 
 and gives some intrinsic characterizations of the $p$-essential normality.  Bounded symmetric domains are non-compact Hermitian symmetric spaces that are natural generalizations of the usual unit ball and play a role in various branches of mathematics, such as arithmetic algebraic geometry, differential geometry, and Lie theory.   The bounded symmetric domain is assumed to be irreducible in this paper unless otherwise stated.   It is known that the Hardy and weighted  Bergman modules are not  $\infty$-essentially normal in the higher rank case due to the foliation structure for the Toeplitz $C^\ast$-algebra \cite{Upm84}, but there indeed exist plentiful nontrivial quotient submodules that are determined by algebraic varieties are $p$-essentially normal for sufficiently large $p;$ see example for \cite{GWZ23,UpW16, Zha21}.  

  Salinas proved in \cite[Corollary 3.2]{Sal89} that $\infty$-essential normality of Bergman modules on bounded domains is preserved under proper holomorphic maps,  provided the maps can be continuously extended to the boundaries.  Since proper holomorphic self-maps on irreducible bounded symmetric domains with higher dimensions must be automorphisms,  it follows that $\infty$-essential normality of  Bergman modules is preserved invariant if the coordinate multipliers are replaced by arbitrary automorphism multipliers. 
  Recently,  Kennedy and Shalit showed in  \cite[Theorem 2.1]{KS12} that the $p$-essential normality preserves invariant under a generalized invertible linear transformation for quotient Drury-Arveson submodules,  provided the submodules are determined by some homogeneous algebraic varieties in the unit ball. 
On the other hand,   the methodology used in \cite{EnE15} enlightens us to develop the $p$-essentially normal theory of Hilbert modules on some complex manifolds.   However, once one attempts to consider the $p$-essentially normal theory on complex manifolds, one has to handle the problem that the $p$-essential normality is independent of the choice of holomorphic coordinate charts.   These observations motivate us to investigate the invariance of $p$-essential normality for more general quotient submodules under biholomorphisms on bounded symmetric domains. 
We shall consider the invariance problem in a more general context of submodules and quotient submodules which are determined by analytic varieties.

To express our main results clearly, we introduce some terminologies and notations.  Let $H$ be a  Hilbert module over the complex polynomial algebra $\mathbb{C}[z]=\mathbb{C}[z_1,\cdots,z_n]$ in $n$ variables and  $T=(T_1,\cdots, T_n)$ be the commuting operator $n$-tuple of the respective coordinate operators,  which is denoted by $(H, T)$ in this paper. A closed subspace $M\subset H$ is called a (Hilbert) submodule if $TM\subset M,$ which is also denoted by $(M, T).$
 The orthogonal complement $M^\bot$ is a submodule  of $(H,S)$, 
where  $S=(P_{M^\bot}T_1,\cdots,P_{M^\bot}T_n)$   is the compression $n$-tuple of  $T.$ Usually, $(M^\bot , S)$ is called a quotient  submodule of $(H,T).$
 A  submodule $(M,T)$ of $(H,T)$ is said to be $p$-essentially normal  if all cross-commutators
\begin{equation}\label{com} [T_i|_M,(T_j|_M)^*]= T_i|_M\hspace{0.1em}(T_j|_M)^*-(T_j|_M)^*\hspace{0.1em}T_i|_M\notag \end{equation} belong to the Schatten $p$-class  $\mathcal{L}^p(M),p\in [1,\infty]$ for all  $i,j=1,\cdots,n,$ where $T_i|_M$ denotes the restriction of $T_i$ to $M$ and $\mathcal{L}^\infty(M)$ is interpreted as the compact operator ideal. Similarly, one can define the $p$-essentially normal quotient submodules.

 Our first main result involves the Taylor functional calculus. Let $\Omega\subset \mathbb{C}^n $ be a bounded symmetric domain. 
  Denote by  $\widehat{S}$ an arbitrary permissive linear transformation for the triple  $(\Omega, M^\bot, S)$  and by $\phi(\widehat{S})$ the Taylor functional calculus for $\widehat{S}.$  A linear transformation for the $n$-tuple $S$ is called permissive if its Taylor spectrum 
  is contained in the closure $\bar{\Omega};$ see (\ref{pers}) for the explicit definition.

\begin{thm1}\label{th1} Suppose $\phi$ is an automorphism on $\Omega.$  Let $(M, T)$ be a submodule of a  Hilbert module $(H, T),$  then the following are equivalent:
\begin{enumerate}
\item $(M^\bot,S)$ is $p$-essentially normal.
\item $(M^\bot,\widehat{S})$ is $p$-essentially normal.
\item $(M^\bot,\phi(\widehat{S}))$ is $p$-essentially normal.
\end{enumerate}
  \end{thm1}
In particular, if the Taylor spectrum $Sp(S)\subset\bar{\Omega}$ then $(M^\bot,S)$ is $p$-essentially normal if and only if $(M^\bot,\phi({S}))$ is $p$-essentially normal. It can be checked that quotient submodules $(M^\bot,\widehat{S})$  and $(M^\bot,\phi(\widehat{S}))$ over the same polynomial algebra are not $2p$-isomorphic in the sense of Arveson in general, hence Theorem 1 differs from the invariance result \cite[Corollary 4.5]{Arv07}.

 Our second main result deals with the biholomorphic invariance of $p$-essential normality of analytic Hilbert modules.     Throughout this paper, we use $(\Omega,w),(\Omega,z)$ to indicate that the bounded symmetric domain $\Omega$ is assigned different coordinates $w,z,$ respectively. 
 Let $(M, T_z)$ be a  weighted Bergman submodule determined by an analytic variety in $(\Omega,z),$  equipped with  $T_z=(M_{z_1},\cdots, M_{z_n}),$ which is the usual coordinate function multipliers. Its quotient Bergman submodule is denoted by  $(M^\bot, S_z)$  where  $S_z$ is the compression of $T_z.$ Clearly, $S_z$ coincides with the truncated Toeplitz operator $n$-tuple  with the symbol of  coordinate functions.
  For any automorphism  $\phi=(\phi_1,\cdots,\phi_n): \Omega\rightarrow \Omega,$   we denote  by
 $\phi^*(M^\bot), S_\phi=(S_{\phi_1},\cdots, S_{\phi_n})$  the pullback of $M^\bot$ under $\phi,$ the truncated Toeplitz operator $n$-tuple with the symbol of the component functions of $\phi,$  respectively. It is known from Section \ref{seon}  that  $\phi^*(M^\bot)=(\phi^*(M))^\bot.$

 \begin{thm2}\label{the} Let    $M$ be a weighted Bergman submodule determined by an analytic variety in $\Omega$.  Suppose that $\phi=(\phi_1,\cdots,\phi_n): (\Omega,w)\rightarrow (\Omega,z) $ is an automorphism with $z=\phi(w)$, then the following hold:
 \begin{enumerate}
\item The $\mathbb{C}[w]$-module $(\phi^*(M^\bot),S_w)$ is $p$-essentially normal if and only if  the $\mathbb{C}[z]$-module  $(M^\bot,S_z)$  is $p$-essentially normal.
\item    The following two statements are equivalent.
\begin{enumerate}
\item $[S_{w_i},S_{w_j}^*]\in \mathcal{L}^p(M^\bot)$ for all $i,j=1,\cdots,n.$
\item $[S_{\phi_i},S_{\phi_j}^*]\in\mathcal{L}^p(M^\bot)$ for all $i,j=1,\cdots,n.$
\end{enumerate}
\end{enumerate}
\end{thm2}

Theorem  2 implies that the $p$-essential normality of a quotient weighted Bergman submodule 
is preserved invariant if the coordinate multipliers are replaced by arbitrary automorphism multipliers. Theorem  2  gives an intrinsic characterization of the $p$-essential normality of Hilbert modules, namely the $p$-essential normality is independent of the choice of the holomorphic coordinate charts.     This ensures basically that one can develop the theory of $p$-essential normality on complex manifolds once there is an appropriate definition of Hilbert modules; further discussion on this issue will appear elsewhere. 
Proposition \ref{MS} enables us to endow the $\mathbb{C}[w]$-module $(\phi^*(M^\bot),S_\phi)$ with a natural $\mathbb{C}[z]$-module structure which is unitarily equivalent to  the $\mathbb{C}[z]$-module  $(M^\bot,S_z)$. However,  the $\mathbb{C}[w]$-module $(\phi^*(M^\bot),S_w)$ does not possess this property in general, namely the change of holomorphic coordinates does not give rise to the unitary isomorphism of quotient submodules. 
With extra effort, similar results on Hilbert modules corresponding to the continuous part of the Wallach set are also obtained, particularly Hardy and Drury-Arveson modules; see Theorem \ref{hwaa}.  


 Moreover, we consider the  Taylor spectrum of compression operator tuple on the quotient submodule. 
 We calculate the Taylor spectrum 
 when the submodule is finitely generated and is determined by an analytic variety in an open neighborhood of $\bar{\Omega};$  see the Proposition \ref{vspo}. 
 As an application,  we prove the solvability of the related corona problem of the quotient  Drury-Arveson submodule for a multiplier subalgebra; 
 see Proposition \ref{corona1}. 
 When $\Omega=\mathbb{B}^n$ the solvability result provides supporting evidence for the conjecture proposed by Davidson, Ramsey, and Shalit  \cite[Remarks 5.5]{DRS15}. 

Finally, we present two direct applications of our invariance results.  Recently,  both the Arveson-Douglas  conjecture and  Arveson's hyperrigidity conjecture have attracted considerable attention; see \cite{CT21, GuW20, KS16} and references therein. In \cite{KS16},   Kennedy and Shalit established a deep connection 
between hyperrigidity for an operator tuple and $\infty$-essential normality, which in fact connects the above two conjectures.  First we provide an equivalent description of the Geometric Arveson-Douglas  conjecture; see Conjecture \ref{cnj1}. As the second application,
we extend the Kennedy-Shalit theorem  \cite[Theorem 4.12]{KS16} to a broader range of compression tuples and quotient submodules, which highlights the equivalence between hyperrigidity and $\infty$-essential normality; see Theorem \ref{snhh}.  

The paper is organized as follows. Section 2 reviews some of the standard facts about analytic varieties and bounded symmetric domains in terms of the Jordan triple Systems.  In Section 3, we give the proof of Theorem 1.  Section 4 is devoted to proving Theorem 2. Section 5 discusses the case of Hilbert modules corresponding to the continuous part of the Wallach set.  In Section 6, we calculate the  Taylor spectrum of compression operator tuple on the quotient submodule, 
and consider the related corona problem for the quotient submodule.  Section 7 contains two direct applications. 

 \section{preliminaries}\label{prel}

 In this section, we will recall some known facts that will be used later on.  
 Let $D\subset \mathbb{C}^n$ be a domain. 
 A nonempty closed subset $A\subset D$ is said to be an analytic variety, if for every $z_0\in A,$ there exists a neighborhood $U$ of $z_0$ and holomorphic functions $f_1,\cdots,f_m\in \mathcal{O}(U)$ such that
 $$A\cap U=\{z\in U: f_1(z)=\cdots=f_m(z)=0\}.$$
Then $f_1,\cdots,f_m$ are said to be local equations of $A$ in $U.$ The definition of the analytic variety is local. A nontrivial fact is that every analytic variety can be globally defined by finitely many holomorphic functions on Stein manifolds \cite[Theorem 5.14]{GF76}, especially on pseudoconvex domains in $\mathbb{C}^n.$ 
As we shall mention below,  bounded symmetric domains are always pseudoconvex, so in the following, we need only consider the globally defined analytic variety in bounded symmetric domains.  It is easy to check that if $f:\Omega_1\rightarrow\Omega_2$ is a holomorphic map between two domains and $A\subset \Omega_2 $ is an analytic variety, so is its preimage $f^{-1}(A)\subset\Omega_1.$

We now briefly recall some known facts about bounded symmetric domains and their Jordan theoretic description; we refer the reader to \cite{Chu21, Din22, Loo77, Upm96}  for general background. 
A bounded symmetric domain is said to be irreducible if it can not be written as a Cartesian product of bounded symmetric domains of lower dimensions up to biholomorphisms. There only exist six types of irreducible bounded symmetric domains up to biholomorphisms due to Cartan, and the others are the  Cartesian product of irreducible bounded symmetric domains up to biholomorphisms. Let $G$ be the identity component of the automorphism group $\text{Aut}(\Omega)$ of $\Omega.$ Then $\Omega = G/K$ is a realization of  Hermitian symmetric space of the non-compact type,   where $K = \{g \in G:  g (\bm{0})=\bm{0}\}$ is the isotropy subgroup of $\bm{0}\in \Omega$.  It is well-known \cite{Loo77, Upm96} that $Z$ can be equipped with a Hermitian Jordan triple structure, and then $\Omega$  can biholomorphically be realized as the spectral unit ball in $Z,$ thus we identify  $\Omega$ with the unit ball in the spectral norm. 
The spectral unit ball in $\mathbb{C}^n$ is a bounded balanced convex domain, especially a pseudoconvex domain. Note that any two norms of the finite-dimensional complex vector space $\mathbb{C}^n$ are equivalent, thus the topologies induced by both the spectral norm and the usual  Euclidean norm coincide on $\mathbb{C}^n.$ Since $\Omega$ is circular, one can show that the isotropy subgroup $K=\{ g\in GL(\mathbb{C}^n): g(\Omega)=\Omega\}.$ We denote the Jordan triple product by $$Z\times Z\times Z\rightarrow Z, \quad (u,v,w) \mapsto \{uv^*w\}$$ which is complex linear in $u,w$ and complex  conjugate   linear in $ v.$  
For each pair $(u,v) \in Z \times Z$ the linear endomorphism
$$ B(u,v)w := w - 2\{uv^*w\} + \{u\{vw^*v\}^*u\}$$
is called the Bergman endomorphism associated with $(u,v)$. One can show that $B(z,\xi) \in GL(Z)$ is invertible whenever $z,\xi\in\Omega.$
 A pair $(z,\xi)\in Z\times Z$ is said to be quasi-invertible if the Bergman endomorphism $B(z,\xi)$ is invertible,  in this case $$z^{\xi}:=B(z,\xi)^{-1}(z-\{z\xi^*z\})$$ is called the quasi-inverse of the pair $(z,\xi);$ see \cite{Loo77,UpW16} for details.  Now fix $z_0\in \Omega,$  we can show that
$$g_{z_0}(w)=z_0+B(z_0,z_0)^{\frac{1}{2}}w^{-z_0}$$ defines an automorphism $g_{z_0}\in\text{Aut}(\Omega)$ called the M\"obius transformation associated with $z_0.$ It can be checked that 
$g_{z_0}$  satisfies $$g_{z_0}(\bm{0})=z_0, g_{z_0}(-z_0)=\bm{0},$$
and its inverse is $g_{-z_0}.$


It is well-known that there exists a unique generic polynomial $\Delta(z,w)$ in $z,\bar{w}$ and a numerical invariant $N$ satisfying the property that the Bergman kernel of  $\Omega$ is given by \begin{equation}\label{ker} K_N(z,w)=\Delta(z,w)^{-N},\end{equation} where the numerical invariant  $N$ is also called the genus of domain $\Omega,$ and the generic polynomial $\Delta(z,w)$ is also called Jordan triple determinant such that $\Delta(0,0)=1$. The $K$-invariant probability measure $dv_\gamma$ on  $\Omega$ is given by $$dv_\gamma(w)=c_\gamma \Delta(w,w)^{\gamma} dv(w)$$
 for $\gamma>-1,$  where $c_\gamma$ is the normalized  constant and  $K$  is the isotropic subgroup as above.  Let $A^2(dv_\gamma)$  denote the weighted Bergman space consisting of square-integrable holomorphic functions with respect to the measure $dv_\gamma$ on $\Omega,$  which is occasionally denoted by $A^2(\Omega,dv_\gamma).$ The Bergman kernel $K_\gamma(z,w)$ of the spectral unit ball is given by
$$ K_{N+\gamma}(z,w)=\Delta(z,w)^{-(N+\gamma)},$$
which is degenerated to the formula (\ref{ker})  when $\gamma=0.$   The weighted Bergman space $A^2(dv_\gamma)$ is a reproducing kernel Hilbert function space, since  \begin{equation}\label{repf} f(z)=\langle f, K_{N+\gamma,z}\rangle_\gamma=\int_{\Omega}f(w)\overline{K_{N+\gamma,z}(w)}dv_\gamma(w),  \quad   z \in \Omega\end{equation}
for every $f\in A^2(dv_\gamma),$ where $$K_{N+\gamma,w}(z)= K_{N+\gamma}(z,w)=\Delta(z,w)^{-(N+\gamma)}, \quad   z,w \in \Omega,$$  is also called the reproducing kernel of $A^2(dv_\gamma).$

\begin{exm}
 Type $\text{\uppercase\expandafter{\romannumeral1}}_{r\times n}.$  Let $Z=\mathbb{C}^{r\times n}$ be the complex  $(r\times n)$-matrix space with  $r\leq n.$ Then
$$\text{\uppercase\expandafter{\romannumeral1}}_{r\times n}=\{u\in Z: \text{Id}_r-uu^*>0\}$$ is an irreducible bounded symmetric domain, where $\text{Id}_r$  is the $(r\times r)$-unit matrix and $u^*$ is the conjugate matrix of $u.$ The associated Jordan triple product on $Z$ is the generalized anti-commutator product $$\{uv^*w\}:=\frac{1}{2}(uv^*w+wv^*u)$$ for $u,v,w\in Z,$ where the multiplication in the right side is the usual matrix multiplication. In this case, the Bergman endomorphism associated with $(u,v)$ is given by
$$B(u,v)w=(\text{Id}_r-uv^*)w(\text{Id}_n-v^*u)$$  and the  quasi-inverse of the pair $(z,\xi)$ is $$z^{\xi}=(\text{Id}_r-z\xi^*)^{-1}z.$$ It follows that the corresponding generic polynomial  $\Delta(u,v)$ is $$\Delta(u,v)=\text{Det}(\text{Id}_r-uv^*).$$
The case of $r=1$ is the unit ball $\text{\uppercase\expandafter{\romannumeral1}}_{1\times n}=\mathbb{B}^n\subset \mathbb{C}^{1\times n}=\mathbb{C}^n$ in the  Euclidean norm, in this case, $\Delta(u,v)=1-\langle u,v\rangle,$ where $\langle \cdot,\cdot\rangle$ is the standard Hermitian inner product on $\mathbb{C}^n.$
\end{exm}

From \cite{Upm96}, we know that  the natural action of $K$ on $\mathcal{P}(Z)=\mathbb{C}[z]$ induces the Peter-Schmid-Weyl decomposition
$$ \mathcal{P}(Z)=\sum_{\bm{m}\geq0}\mathcal{P}_{\bm{m}}(Z),$$
 where $\bm{m}=( m_1,\cdots , m_r)\geq0$
runs over all integer partitions, namely $$ m_1 \geq\cdots \geq m_r\geq 0.$$ The decomposition is irreducible under the action of $K$ and is orthogonal under the  Fischer-Fock (or Segal-Bargmann) inner product $\langle \cdot,\cdot\rangle_F;$ see \cite[Section 2.7]{Upm96}. For a $r$-tuple $\bm{s}=(s_1,\cdots,s_r)\in \mathbb{C}^r,$ the Gindikin Gamma function \cite{Din22,Upm96} is given by $$\Gamma_\Omega(\bm{s})=(2\pi)^{\frac{ar(r-1)}{4}}\prod_{j=1}^r\Gamma(s_j-\frac{a}{2}(j-1)),$$
of the usual Gamma function $\Gamma$ whenever the right side is well defined, where $r$ is the rank of $\Omega$ and $a,b$ are two numerical invariants (nonnegative integers) associated with the joint Peirce decomposition for a chosen frame $e_1,\cdots,e_r$ of minimal tripotents such that the dimension count
\begin{equation}\label{cont}n=r+\frac{a}{2}r(r-1)+br\end{equation} holds and the genus $N$ is given by $$N:=2+a(r-1)+b.$$ The multi-variable Pochhammer symbol is
$$(\lambda)_{\bm{s}}:=\frac{\Gamma_\Omega(\lambda+\bm{s})}{\Gamma_\Omega(\bm{s})},$$ where $\lambda+\bm{s}:=(\lambda+s_1,\cdots,\lambda+s_r).$
And the $G$-orbit $$\mathcal{S}:=G\cdot(e_1+\cdots +e_r)$$  is called the Shilov boundary of  $\Omega,$ which is a compact analytic manifold contained in the topological boundary $\partial \Omega.$ In general the topological boundary  $\partial \Omega$ is not smooth. 
It can be verified that \begin{equation}\label{pola} (\lambda)_{\bm{s}}= \prod_{j=1}^r(\lambda-\frac{a}{2}(j-1))_{s_j}\end{equation} of the usual Pochhammer symbols $(\mu)_m=\prod_{j=1}^m(\mu+j-1).$
It is known from \cite{AZ03,Upm96} that
\begin{equation}\label{dkl} \Delta(z,w)^{-\lambda}=\sum_{\bm{m}\geq0}(\lambda)_{\bm{m}} K_{\bm{m}}(z,w)\end{equation}
converges compactly and absolutely on  $ \Omega\times \Omega,$ where $K_{\bm{m}}$ is the reproducing kernel of $\mathcal{P}_{\bm{m}}(Z)$ in the Fischer-Fock inner product, for all $\lambda\in\mathbb{C}.$
 The Wallach set $W_\Omega$ with respect to $\Omega$  is defined to be the set consists of all $\lambda\in \mathbb{C}$ satisfying  $$(\lambda)_{\bm{m}}\geq0$$ for all  integer partitions $\bm{m}\geq0,$ which is exactly the set of value $\lambda$ such that   $\Delta(z,w)^{-\lambda}$ is a positive kernel. The Wallach set $W_\Omega$ admits
 the following decomposition $$W_\Omega=W_{\Omega,d}\cup W_{\Omega,c}$$
 where $W_{\Omega,d}=\{\lambda=(j-1)\frac{a}{2},j=1,\cdots,r\}$ and  $W_{\Omega,c}=\{\lambda>(r-1)\frac{a}{2}\}.$

 The set $W_{\Omega,c}$ is called the continuous part, in this  case, the function $\Delta(z,w)^{-\lambda}$ is the reproducing kernel of the Hilbert holomorphic function space
 $$H_{\lambda}^2(\Omega)=\sum_{\bm{m}\geq0} \mathcal{P}_{\bm{m}}(Z),$$
whose inner product is induced by \begin{equation}\label{pqin}(p,q)_\lambda=\frac{1}{(\lambda)_{\bm{m}}}\langle p,q\rangle_F\end{equation}for all $p,q\in  \mathcal{P}_{\bm{m}}(Z),\bm{m}\geq0.$ The weighted Bergman space $A^2(dv_\gamma)$ coincides with $H^2_{N+\gamma}(\Omega),\gamma>-1$ and the classical Hardy space  $H^2(\Omega)$ defined in \cite[Definition 2.8.4]{Upm96} coincides with $H^2_{\frac{n}{r}}(\Omega).$ When $\Omega=\mathbb{B}^n$  we know that  $H^2_1(\mathbb{B}^n)$ coincides with the Drury-Arveson space defined in \cite{Arv98}. Thus in general we define the Drury-Arveson space to be $H^2_{(r-1)\frac{a}{2}+1}(\Omega).$
 The set $W_{\Omega,d}$ is called the discrete part, in the case of $\lambda=(j-1)\frac{a}{2},$  the function $\Delta(z,w)^{-\lambda}$ is the reproducing kernel of the Hilbert holomorphic function space
 $$H_{\lambda}^2(\Omega)=\sum_{\bm{m}\geq0,m_j=0} \mathcal{P}_{\bm{m}}(Z),$$
 whose inner product is defined similarly to (\ref{pqin}).
The Drury-Arveson space $H^2_1(\mathbb{B}^n)$  on  $\mathbb{B}^n$ is complete Nevanlinna-Pick \cite{DRS15}; actually, the reproducing kernel Hilbert holomorphic function space $H_{\lambda}^2(\mathbb{B}^n)$  is complete Nevanlinna-Pick if and only if $0<\lambda\leq1.$ On the contrary, in the higher rank case of $r\geq2,$  one can show that $H_{\lambda}^2(\Omega)$ is not complete Nevanlinna-Pick for every $\lambda\in W_{\Omega},$ by the Agler-McCarthy characterization of complete Nevanlinna-Pick kernels.

 \section{The proof of Theorem 1}\label{Sec}
In this section, we will prove Theorem 1.  
For the special case $p=\infty,$ the proof is relatively simple since the compact operator ideal is norm-closed; see details in \cite{Sal89}. However, the Schatten $p$-class ideal has no such a property for general $p\in (0,\infty].$ Our proof for the general case is inspired by Connes'  approach \cite{Con85} to dealing with the Dunford-Riesz functional calculus of Schatten class commutators. 
Nevertheless, the Taylor functional calculus for a commuting operator tuple is originally proved by means of the Cauchy-Weil formula for cohomology class and the formula is difficult to apply in practice. To handle these difficulties we first establish new integral formulas with the rational function (Bergman and Szeg\"o) kernels for Taylor functional calculus by the holomorphic function space theory.

The following two operator identities are very elementary and can be immediately verified by direct calculations.
\begin{lem}\label{age} Let $U, V, W$ be bounded operators on a Banach space, then the following hold:
\begin{enumerate}
\item  $ [UV,W]=U[V,W]+[U,W]V.$
\item If $U$ is invertible, then $$[U^{-m},V]=U^{-m}[V,U^m]U^{-m}=U^{-(m+1)}[U,V]-U^{-(m+1)}[U^{m+1},V]U^{-m},$$
for every integer $m\geq0.$
\end{enumerate}
\end{lem}
For a holomorphic polynomial $q\in \mathbb{C}[z],$ we define $S_q$ by \begin{equation}\label{sps} S_q:=q(S_1,\cdots,S_n)=P_{M^\bot}q(T_1,\cdots,T_n).\end{equation}
For quotient submodule $(M^\bot,S)$ of $(H,T)$, it is evident that $S_{z_i}=S_i,i=1,\cdots,n.$
\begin{prop}\label{Sij} The following are  equivalent:
\begin{enumerate}
\item $[S_{i},S_{j}^*]\in \mathcal{L}^p(M^\bot)$ for all $i,j=1,\cdots,n.$
\item $[S_{z_i},S_{z_j}^*]\in \mathcal{L}^p(M^\bot)$ for all $i,j=1,\cdots,n.$
\item $[S_{h},S_{q}^*]\in\mathcal{L}^p(M^\bot)$ for all $h,q\in \mathbb{C}[z].$
\end{enumerate}
\end{prop}
\begin{proof} This  follows from  Lemma \ref{age}\hspace{0.1em}(1) and  induction on the degree of  polynomials.
\end{proof}

For a bounded commuting operator $n$-tuple $T=(T_1,\cdots, T_n)$ on a Hilbert space $H,$ we denote by $Sp(T)$  its Taylor spectrum which is a nonempty compact subset of the closed polydisc $\Delta_T=\{z\in \mathbb{C}^n:\vert z_i\vert\leq r( T_i ),i=1,\cdots,n\},$ where $r(T_i)$ 
 is the spectral radius of $T_i.$ 
Let $B(H)$ denote the set of bounded linear operators and $\mathcal{O}(Sp(T))$ denote the set of holomorphic functions on some neighborhoods of  the compact set $Sp(T),$ then there exists a unique continuous algebraic homomorphism \begin{equation}\begin{split}\label{}\mathcal{O}(Sp(T))&\rightarrow (T)''\subset B(H), \notag\\
f&\longmapsto f(T),
\end{split}\end{equation} which satisfies
$$\bm{1}(T)=\text{Id}, z_i(T)=T_i,i=1,\cdots,n,$$ and  the following spectral mapping theorem, where $(T)''$ is the bicommutant algebra of   $T.$ This is known as the  Taylor functional calculus \cite{Tay70A}. Suppose $f=(f_1,\cdots,f_m):\Omega\rightarrow\mathbb{C}^m$ is a holomorphic map and $Sp(T)\subset \Omega.$  Let $f(T)=(f_1(T),\cdots,f_m(T)),$  then the spectral mapping theorem holds:
$$f(Sp(T))=Sp(f(T)).$$ The existence of the Taylor functional calculus is proved in \cite{Tay70A} by means of Cauchy-Weil formulas and its proof is rather involved.  In what follows, we provide two simpler integral formulas of the Taylor functional calculus concerning rational function kernels.  We first briefly recall the definition of the Hardy space.  The Hardy space $H^2(\Omega)$ is consists of  all holomorphic functions on $\Omega$ satisfying
$$\sup_{R\rightarrow 1^-}\int_{\mathcal{S}}\vert f(Rw)\vert^2d\sigma(w)<\infty,$$
 where $\mathcal{S}$ is the Shilov boundary of $\Omega$ with normalized $K$-invariant Haar measure $d\sigma.$   There exists a canonical unitary isomorphism between $H^2(\Omega)$ and a closed subspace $H^2(\mathcal{S})$ of  $L^2(\mathcal{S},d\sigma);$  see \cite[Corollary 2.8.47]{Upm96} for details. Thus we  identify  $H^2(\Omega)$ with  $H^2(\mathcal{S}).$ The Szeg\"o kernel of $\Omega$ is given by
$$S(z,\xi)=\Delta(z,w)^{-\frac{n}{r}}, \quad (z,w)\in \Omega\times \mathcal{S}.$$  Let $\mathcal{O}(\bar{\Omega})$ be the set of holomorphic functions on some neighborhoods of $\bar{\Omega}.$ Then the following integral formulas of the Taylor functional calculus hold.
\begin{lem}\label{fhtt}
Suppose $Sp(T)\subset \Omega$  and $f\in \mathcal{O}(\bar{\Omega}),$ then the following integral formulas hold:
  \begin{equation}\label{int} f( T)=\int_{\mathcal{S}}\frac{f(w)}{\Delta(T,w)^{\frac{n}{r}}}d\sigma(w),\end{equation}
  \begin{equation}\label{inttt} f( T)=\int_{\Omega}\frac{f(w)}{\Delta(T,w)^{N}}dv(w).\end{equation}
\end{lem}
\begin{proof} By the spectral mapping theorem, we have  $Sp \,\Delta(T,w)= \Delta(Sp(T),w),$ for every $w\in \mathcal{S}\subset \partial \Omega.$  It follows from \cite[Section 7.5]{Loo77} that $ \Delta(z,\xi)\neq 0$ for  the pair $(z,\xi)\in Z\times Z$ whenever the spectral norm satisfying $\Vert z\Vert\cdot\Vert \xi\Vert<1.$ Then  $\bm{0}\notin Sp \,\Delta(T,w),$ and hence the map $\Delta(T,w)$ is invertible on the Shilov boundary $\mathcal{S}.$ Combining  the formula (\ref{dkl}) with the uniqueness of the Taylor expansion of analytic functions follows  that \begin{equation} \Delta(z,w)^{-\lambda}=\sum_{\bm{m}\geq0}(\lambda)_{\bm{m}} K_{\bm{m}}(z,w)\notag\\\end{equation}
converges compactly and absolutely on the circular domain  $ t\Omega\times t^{-1}\Omega$ for every $t>0,$ where $t\Omega=\{tz:z\in \Omega \}.$ Thus, by the continuity and uniqueness of the Taylor functional calculus, it follows that for all $\lambda\in \mathbb{R},$ \begin{equation}\label{delT} \Delta(T,w)^{-\lambda}=\sum_{\bm{m}\geq0}(\lambda)_{\bm{m}} K_{\bm{m}}(T,w)\end{equation} for every $w\in \mathcal{S},$ and the convergence is uniform on $\mathcal{S}$ since the compactness of $\mathcal{S}.$ Let $f=\sum_{\bm{m}\geq \bm{0}} f_{\bm{m}}$  be the Peter-Schmid-Weyl decomposition of $f.$  Clearly, the convergence is compactly on $\Omega\cup \mathcal{S}$ and the polynomials $f_{\bm{m}}$ satisfying $$f_{\bm{m}}(z)= \int_\mathcal{S} f(w)(\frac{n}{r})_{\bm{m}} K_{\bm{m}}(z,w)d\sigma(w)$$ for all $\bm{m}\geq \bm{0}.$
Thus  \begin{equation}\begin{split}\label{}
f(T)&=\sum_{\bm{m}\geq \bm{0}} f_{\bm{m}}(T) \notag\\
     &=\sum_{\bm{m}\geq \bm{0}} \int_\mathcal{S} f_{\bm{m}}(w)(\frac{n}{r})_{\bm{m}} K_{\bm{m}}(T,w)d\sigma(w)\\
     &=\sum_{\bm{m}\geq \bm{0}} \int_\mathcal{S} f(w) (\frac{n}{r})_{\bm{m}} K_{\bm{m}}(T,w)d\sigma(w)\\
      &=\int_\mathcal{S} f(w) \sum_{\bm{m}\geq \bm{0}} (\frac{n}{r})_{\bm{m}} K_{\bm{m}}(T,w)d\sigma(w)\\
      &=\int_\mathcal{S}\frac{f(w)}{\Delta(T,w)^{\frac{n}{r}}}d\sigma(w).
\end{split}\end{equation}
This proves the formula (\ref{int}).  Similarly, the formula (\ref{inttt}) can be proved.
\end{proof}

\begin{rem} 
If the compact set $\Delta_T\subset \Omega,$ then the Dunford-Riesz functional calculus for a single operator is enough for our purposes. We now give a brief explanation. Since $Sp(T)\subset \Delta_T\subset\Omega,$ there exist  polydiscs $ \Delta_{\bm{r}}=\{z\in\mathbb{C}^n:\vert z_i\vert<r_i,i=1,\cdots,n\}$ and $\Delta_{\bm{r}'}=\{z\in\mathbb{C}^n:\vert z_i\vert<r_i',i=1,\cdots,n\}$ such that $$\Delta_T\subsetneq\Delta_{\bm{r}}\subset\bar{\Delta}_{\bm{r}}\subset\Delta_{\bm{r}'}\subset\bar{\Delta}_{\bm{r}'}\subset\Omega.$$ Since a holomorphic function on a circular domain  containing the origin has a homogeneous holomorphic polynomial expansion which converges compactly,  it follows from the continuity and uniqueness of the Taylor functional calculus that
 \begin{equation}\label{intt} f( T)=\frac{1}{(2\pi\sqrt{-1})^n}\int_{\vert \xi_1\vert=r_1}\cdots\int_{\vert \xi_n\vert=r_n} \frac{d\xi_1}{\xi_1-T_1}\cdots \frac{d\xi_n}{\xi_n-T_n}f(\xi)\notag\\\end{equation}
for $f\in \mathcal{O}(\Omega),$ where the right side repeatedly uses the Cauchy integral formula.
\end{rem}

In the following, we establish a local extension property of the automorphism on bounded symmetric domains. Let $ \Omega_1,\Omega_2$ be two domains,  we say that $\phi: \bar{\Omega}_1\rightarrow \Omega_2 $ is holomorphic if $\phi$ is the restriction of a holomorphic map from an open neighborhood of $\bar{\Omega}_1$ to $ \Omega_2,$ namely there exists an open set $U\supset\bar{\Omega}_1$ and a holomorphic map $\tilde{\phi}: U\rightarrow  \Omega_2$ satisfying $\tilde{\phi}\vert_{\Omega_1}=\phi.$  We also say  that  $\phi: \bar{\Omega}_1\rightarrow \bar{\Omega}_2 $ is biholomorphic if $\phi$ is the restriction of a biholomorphism between neighborhoods of $\bar{\Omega}_1, \bar{\Omega}_2.$ The following lemma shows that every automorphism on $\Omega$ is indeed also a biholomorphism on $\bar{\Omega},$ which generalizes the fact \cite[Remark 3.2.6]{Chu21} that every Mobius transformation $g_{z_0}$ is indeed  biholomorphic on $\bar{\Omega}$ for $z_0\in\Omega.$
 
\begin{lem}\label{bihg} If $\phi\in\text{Aut}(\Omega),$ then  $\phi:\bar{\Omega}\rightarrow\bar{\Omega}$ is biholomorhphic.
 \end{lem}
 \begin{proof}  Without loss of generality, we assume that  $\phi(\bm{0})=z_0.$  Note that $g_{z_0}(w)=z_0+B(z_0,z_0)^{\frac{1}{2}}w^{-z_0}$ satisfying $g_{z_0}(\bm{0})=z_0,$  where $z^{\xi}=B(z,\xi)^{-1}(z-\{z\xi^*z\})$ is the quasi-inverse of the pair $(z,\xi)\in \Omega\times\Omega.$  Thus $g_{z_0}^{-1}\circ\phi \in \text{Aut}(\Omega)$ and $g_{z_0}^{-1}\circ\phi(\bm{0})=\bm{0},$ which implies that $g_{z_0}^{-1}\circ\phi\in K.$ This indeed shows that the  automorphism $\phi$ admits the factorization $\phi=g_{z_0}\circ k$ where $k\in K.$  Suppose $z_0\neq\bm{0},$ define the open spectral ball  $\Omega_{z_0}$ by $\Omega_{z_0}:=\{\Vert z\Vert<\frac{1}{\Vert z_0\Vert}\}\supset\bar{\Omega}.$ Since $k\in K\subset GL(Z)$ is linear and $k(\Omega)=\Omega,$  we deduce that $k(\Omega_{z_0})=\Omega_{z_0}.$
 It suffices to show that  $g_{z_0}$ is biholomorphic  on $\Omega_{z_0}=\{\Vert z\Vert<\frac{1}{\Vert z_0\Vert}\}\supset\bar{\Omega}$ by the holomorphic version of the inverse function theorem. To show that  $g_{z_0}$ is biholomorphic  on $\Omega_{z_0},$ it is enough to show that the Jacobi matrix $g_{z_0}'$ of $g_{z_0}(z)=z_0+B(z_0,z_0)^{\frac{1}{2}}z^{-z_0}$ is nondegenerate on $\Omega_{z_0}.$ By \cite[Lemma 3.2.25]{Chu21},  we conclude that $$g_{z_0}'(z)=B(z_0,z_0)^{\frac{1}{2}}B(z,-z_0)$$ for $z\in\Omega_{z_0}.$
Then the quasi-invertibility of $z^{-z_0}$ on $\Omega_{z_0}$ implies that  $g_{z_0}'$  is nondegenerate on $\Omega_{z_0}.$
The remaining case of $z_0=\bm{0}$ is trivial, since in this case $\phi=k\in K,$ which is biholomorphic on every open set $t\Omega=\{tz:z\in\Omega\}\supset\bar{\Omega}$ for $t>1.$
\end{proof}

 The composition law for Taylor functional calculus was proved in  \cite{Put82}, in particular,  the following composition law on bounded symmetric domains holds.
\begin{lem}\label{fgc} Suppose that $ Sp(T)\subset\Omega_1 $ and   $g:\bar{\Omega}_1\rightarrow\Omega_2$ is a holomorphic map between two bounded symmetric domains. If $f:\bar{\Omega}_2\rightarrow\Omega_3$ is a holomorphic map,  then $$(f\circ g)(T)=f(g(T)).$$
\end{lem}

For every $c>0$ and $\bm{d}\in \mathbb{C}^n,$ they determine a linear transformation of the operator $n$-tuple $T=(T_1,\cdots, T_n)$ on the Hilbert space $H$ in the form of $cT+\bm{d},$  such a linear transformation is called permissive for the triple $(\Omega, H, T),$ if the compact set  \begin{equation}\label{pers} Sp(cT+\bm{d})=c\cdot Sp(T)+\bm{d}\subset \bar{\Omega}.\end{equation} For a given triple $(\Omega, H,T),$ there  exist infinitely many permissive  linear transformations. In the sequel, we denote an arbitrary permissive linear transformation for a given triple $(\Omega, H, T)$ by $\widehat{T}=(\widehat{T}_1,\cdots,\widehat{T}_n).$ We are now in a position to prove Theorem 1.

{\noindent{\bf{Proof of  Theorem 1.}}
 Suppose $\widehat{S}=cS+\bm{d},$  where $c>0$ and $\bm{d}=(d_1,\cdots,d_n)\in \mathbb{C}^n.$ Then  $$[\widehat{S}_i,\widehat{S}_j^*]=[ c{S}_i+d_i,c{S}_j^*+\bar{d}_j]=c^2[S_i,S_j^*],$$
for all $i,j=1,\cdots,n.$ Hence the equivalence between (1) and (2) follows.

It remains to prove the equivalence of  (2) and (3). Let $\phi=(\phi_1,\cdots,\phi_n).$ We first prove that (2) implies (3). The proof will be divided into two cases.
The first case is  $Sp(\widehat{S})\cap\partial\Omega=\emptyset,$ namely,  $Sp(\widehat{S})\subset \Omega.$ Since $\Omega$ is a bounded symmetric domain with rank $r,$ it follows from Lemma \ref{bihg} that every component function $\phi_i$ of the automorphism $\phi$ belongs to $\mathcal{O}(\bar{\Omega}).$ Applying the integral formula of Taylor functional calculus (\ref{inttt}) to holomorphic functions $\phi_i\in \mathcal{O}(\bar{\Omega}),$ it follows that
   \begin{equation}\label{iint} \phi_i( \widehat{S})=\int_{\Omega}\phi_i(z)\Delta(\widehat{S},z)^{-N}dv(z),\end{equation}
where $\Delta(\widehat{S},z)$ is invertible in $B(M^\bot)$ for $z\in \bar{\Omega}.$ 
Since $N$ is a positive integer, it follows  from Lemma \ref{age}\hspace{0.1em}(2) that
\begin{equation}\begin{split}\label{}
[\Delta(\widehat{S},z)^{-N},\widehat{S}_j^*]&=\Delta(\widehat{S},z)^{-N}[\Delta(\widehat{S},z)^{N},\widehat{S}_j^*]\Delta(\widehat{S},z)^{-N}. \\
\end{split}\end{equation}
Since $\Delta(\widehat{S},z)$ is a polynomial of the  $\widehat{S},$  it follows from Proposition \ref{Sij} that the map
 $z\longmapsto \phi_i(z)[\Delta(\widehat{S},z)^{-N},\widehat{S}_j^*]\in \mathcal{L}^p(M^\bot)$ is continuous on $\bar{\Omega},$  for $i=1,\cdots,n.$ Thus $[\phi_i( \widehat{S}),\widehat{S}_j^*]\in \mathcal{L}^p(M^\bot)$ by the formula (\ref{iint}), for all $i,j=1,\cdots,n.$  
 This implies that
$$[\widehat{S}_j,\phi_i( \widehat{S})^*]=[\phi_i( \widehat{S}),\widehat{S}_j^*]^*\in \mathcal{L}^p(M^\bot).$$ Similarly, by using the integral formula of Taylor functional calculus (\ref{inttt}) again, we conclude  that  $$[{\phi_i}(\widehat{S}),\phi_j(\widehat{S})^*]=[\phi_j( \widehat{S}),\phi_i( \widehat{S})^*]^*\in \mathcal{L}^p(M^\bot),$$
for all $i,j=1,\cdots,n.$
The second case is $Sp(\widehat{S})\cap\partial\Omega\neq\emptyset.$ This case can be reduced to the first case. Since
every component function of the automorphism $\phi=(\phi_1,\cdots,\phi_n)$ belongs to $\mathcal{O}(\bar{\Omega}),$ it follows that there exists an open set $D$ which contains the compact set $\bar{\Omega}$ such that $\phi$ is defined  on $D.$ Recall that the open set $\Omega$ is the unit ball in the spectral norm and  the topology induced by the spectral norm  is coincide with the topology induced by the Euclidean norm,  it follows  that  there exists a $\delta>0$ satisfying \begin{equation}\label{exto}\Omega\subsetneq (1+\delta)\Omega\subsetneq (1+\delta)\bar{\Omega}\subsetneq D.\end{equation}
Let $\phi_{(\delta)}=(\phi_{1,\delta},\cdots,\phi_{n,\delta})\in \mathcal{O}(\bar{\Omega})\otimes\mathbb{C}^n$ defined by $$\phi_{i,\delta}(w)=\phi_{i}((1+\delta)w),\quad \forall w\in \Omega,$$
$i=1,\cdots,n.$ Put $ \widehat{S}_{(\delta)}=(\widehat{S}_{1,\delta},\cdots,\widehat{S}_{n,\delta})=\frac{1}{1+\delta}\widehat{S}.$
 By the spectral mapping theorem and (\ref{exto}),  we have  $$Sp( \widehat{S}_{(\delta)})\subseteq \frac{1}{1+\delta}\bar{\Omega}\subsetneq {\Omega},$$ which means that $\widehat{S}_{(\delta)}=\frac{1}{1+\delta}\widehat{S} $ is a permissive  linear transformation of  $\widehat{S}.$
 Thus the equivalence of  (1) and (2)  implies that $(M^\bot,  \widehat{S}_{(\delta)})$ is $p$-essentially normal. From the first case, it follows that $$[\phi_i( \widehat{S}),\phi_j( \widehat{S})^*]=[\phi_{i,\delta}(\widehat{S}_{(\delta)}),\phi_{j,\delta}(\widehat{S}_{(\delta)})^\ast]\in\mathcal{L}^p(M^\bot),$$ for all $i,j=1,\cdots,n.$ Hence (2) implies (3).

Similarly, the same argument for $\phi^{-1}$ and  $\phi(\widehat{S})=(\phi_1(\widehat{S}),\cdots,\phi_n(\widehat{S}))$ follows that
 $[{\phi^{-1}}_i(\phi(\widehat{S})),{\phi^{-1}}_j(\phi(\widehat{S}))^*]\in \mathcal{L}^p(M^\bot)$  if $[{\phi_i}(\widehat{S}),\phi_j(\widehat{S})^*]\in\mathcal{L}^p(M^\bot).$ Then Lemma \ref{fgc} shows that
$$[\widehat{S}_i,\widehat{S}_j^*]=[{\phi^{-1}}_i(\phi(\widehat{S})),{\phi^{-1}}_j(\phi(\widehat{S}))^*]\in\mathcal{L}^p(M^\bot),$$ for all $i,j=1,\cdots,n.$
This proves that  (3) implies (2),  and the proof is complete.
\qed

\begin{exm} (1) Let $\Omega$ be the Lie ball in $\mathbb{C}^n, n\geq3.$ 
For the  inner function $\eta(w)=\sum_{i=1}^nw_i^2,$ denote $\mathcal{Q}=H^2(\Omega) /\eta H^2(\Omega)$ by the corresponding Beurling-type quotient submodule. It follows from \cite{Zha21} that $( \mathcal{Q},S_w)$ is $\infty$-essentially normal and  $Sp(S_w)\subset \bar{\Omega}.$   Let $\phi: \Omega\rightarrow \Omega$ be an automorphism, it follows from Lemma \ref{Hoo} that $\phi(S)=S_\phi,$  the truncated Toeplitz operator $n$-tuple. Thus Theorem 1  implies that 
$[S_{\phi_i},S_{\phi_j}^*]\in\mathcal{L}^\infty(\mathcal{Q})$ for all $i,j=1,\cdots,n.$ 

(2) Let $\Omega\subset \mathbb{C}^n$ be a bounded symmetric domain with rank $r>2$ and  $Z_1=\{z\in\mathbb{C}^n:\text{rank} \hspace{0.1em}(z)\leq1\}$ be a Kepler variety.  
Denote  $ I_{Z_1}$ by the corresponding radical homogenous polynomail ideal 
and $\mathcal{Q}=H^2(\Omega)/\text{Cl}[I_{Z_1}]$ by the quotient submodule, where $\text{Cl}[I_{Z_1}]$ is the closure of $I_{Z_1}.$ It follows from \cite[Theorem 4.2 and Proposition 4.4]{UpW16} that  $(\mathcal{Q},S_w)$ is  $p$-essentially normal for all $p>N,$ where $N$ is the genus. Thus Theorem 1  implies that $[S_{\phi_i},S_{\phi_j}^*]\in\mathcal{L}^p(\mathcal{Q})$ for all $p>N$ and $i,j=1,\cdots,n.$ 
\end{exm}

Similarly, the same is also true for the submodule $(M, T).$

\begin{cor}\label{ccr}

Suppose $\phi$ is an automorphism on  $\Omega.$  Let $(M, T)$ be a submodule of a  Hilbert module $(H, T),$  then the following are equivalent:
\begin{enumerate}
\item $(M,T)$ is $p$-essentially normal.
\item $(M,\widehat{T})$ is $p$-essentially normal.
\item $(M,\phi(\widehat{T}))$ is $p$-essentially normal.
\end{enumerate}
\end{cor}

\section{The proof of Theorem 2}\label{seon}
This section is mainly devoted to the proof of Theorem 2. 
The seed of our main idea lies in the following observations. A closed subspace of the  Bergman space on a bounded symmetric domain is a  submodule over the polynomial algebra if and only if it is a submodule over a Noetherian holomorphic function algebra.  This implies that the categories of Bergman modules over these two function algebras are equivalent; see Proposition \ref{CO} and Remark \ref{Steh}  for a more explicit description.  Since the $p$-essential normality problem we are concerned with is independent of the choice of the above two module categories, we focus on studying $p$-essential normality within the framework of  Bergman modules over the holomorphic function algebras rather than polynomial algebras. One significant benefit arising from this change in module categories is that we no longer need to consider how analytic varieties behave near boundaries.

Let $\mathcal{O}(\bar{\Omega})$ be the set of holomorphic functions on some neighborhoods of $\bar{\Omega}.$ We now briefly explain why  $\mathcal{O}(\bar{\Omega})$ is called  an algebra. Two functions $f_1,f_2\in \mathcal{O}(\bar{\Omega})$ are said to be equivalent if there exists an open neighborhood $D\supset\bar{\Omega}$ such that the restrictions of $f_1,f_2$ on $D$ are equal, i.e., $f_1\vert_{D}=f_2\vert_{D},$  which is denoted by $f_1\sim f_2$ and its equivalence class is denoted by $[f_1].$ Suppose $f_1$ is holomorphic on $D_1\supset \bar{\Omega}$  and $f_2$ is holomorphic on $D_2\supset \bar{\Omega},$ where $D_1,D_2$ are open. We can define three binary operations for equivalence classes as follows:
\begin{equation}\begin{split}\label{Ope}
&[f_1]+[f_2]:=[f_1\vert _{D_1\cap D_2}+f_2\vert _{D_1\cap D_2}],\\
&[f_1]\cdot[f_2]:=[f_1\vert _{D_1\cap D_2}\cdot f_2\vert _{D_1\cap D_2}],\\
&c\cdot [f_1]:= [cf_1], \quad\forall c\in \mathbb{C}.
\end{split}\end{equation}
It is clear that the definitions in (\ref{Ope}) are well-defined. Thus ${\mathcal{O}(\bar{\Omega})}/{\sim}$ becomes a complex unit algebra under the operations in  (\ref{Ope}). Note that  $[\bm{1}]$ is its unit where $\bm{1}$ is  the constant function valued $1$ in a neighborhood of $\bar{\Omega}.$  We see that $ [f]$ is invertible in ${\mathcal{O}(\bar{\Omega})}/{\sim}$  if and only if $ [f]$ has a representative which has no zeros in $\bar{\Omega},$ and its inverse is denoted by $[f]^{-1}.$  By abuse of notation, we continue to write $f,{\mathcal{O}(\bar{\Omega})}$ instead of $[f],{\mathcal{O}(\bar{\Omega})}/{\sim},$ respectively, and we also write $f^{-1}$ instead of $[f]^{-1}$ if $ [f]$ is invertible.  
The Noetherianess of ${\mathcal{O}(\bar{\Omega})}$ comes from 
a theorem due to Frisch and  Siu. We can equip the weighted Bergman space  $A^2(dv_\gamma)$ with a natural  $\mathcal{O}(\bar{\Omega})$-module structure by the multiplication of functions, namely $$ f\cdot h:=f\vert_{\Omega} h,$$
for every $f\in \mathcal{O}(\bar{\Omega})$ and $h\in A^2(dv_\gamma),$ which is well-defined.
Let $M$ be  an  $\mathcal{O}(\bar{\Omega})$-submodule and $M^\bot$ be its orthogonal complement, we define the truncated Toeplitz   operator $S_f:M^\bot\rightarrow M^\bot$ with symbol $f\in\mathcal{A}(\Omega)$  by
\begin{equation}\label{abmul} S_f h:=P_{M^\bot} M_fh=P_{M^\bot} (fh),\end{equation} for every $h\in M^\bot,$ where $\mathcal{A}(\Omega)=C(\bar{\Omega})\cap\mathcal{O}(\Omega)$ is the set of holomorphic functions that can be continuously extended to the boundary  $\partial\Omega,$  $M_f$ is the multiplication operator on $A^2(dv_\gamma)$ with symbol $f,$ and $P_{M^\bot}: A^2(dv_\gamma)\rightarrow M^\bot$ is the orthogonal projection. This coincides with the definition (\ref{sps}) when $T=(M_{z_1},\cdots,M_{z_n})$ is  the coordinate multipliers and $f$ is a polynomial, and in this case $S_i=S_{z_i},i=1,\cdots,n.$ The following lemma implies that $M^\bot$ is an  $\mathcal{O}(\bar{\Omega})$-submodule, whose  module actions are given by $$f\cdot h:=S_f h,$$ for every $ h\in M^\bot$ and $ f\in\mathcal{O}(\bar{\Omega}).$
\begin{lem}\label{fgo} Suppose $f,g\in\mathcal{O}(\bar{\Omega}),$ then the following operator identity holds on $M^\bot:$ $$S_{fg}=S_fS_g.$$
\end{lem}
\begin{proof} Suppose $h\in M^\bot,$ by definition we obtain
\begin{equation}\begin{split}\label{}
S_fS_gh&=P_{M^\bot} (fP_{M^\bot} (gh))\notag\\
&=P_{M^\bot} (f((\text{Id}-P_{M}) (gh)))\\
&=P_{M^\bot} (fgh)-P_{M^\bot} (f(P_{M} (gh)))\\
&=S_{fg}h,\\
\end{split}\end{equation}
where $P_{M^\bot}: A^2(dv_\gamma)\rightarrow M^\bot$  is the orthogonal projection. The third identity holds because that $M$ is an $\mathcal{O}(\bar{\Omega})$-module. This completes the proof.
\end{proof}
\begin{cor}\label{fos} If $f\in\mathcal{O}(\bar{\Omega})$ and has no zeros in $\bar{\Omega},$ then $S_f$ is  invertible on $M^\bot$ and its inverse is
$$ S_f^{-1}=S_{f^{-1}}. $$
\end{cor}
\begin{proof} By the definition, we can assume  that  $f$ is holomorphic on an open set $D\supset\bar{\Omega}.$ Since the set $D_f=\{z\in D: f(z)\neq 0\}$ is  open and the assumption means that $\bar{\Omega}\subset D_f,$ it follows that $f$ is invertible
on $D_f$ and its inverse $f^{-1}\in\mathcal{O}(\bar{\Omega}).$ Combining this with Lemma \ref{fgo}, the desired operator identity follows.
\end{proof}
 It is clear that a closed subspace of $A^2(dv_\gamma)$ is a $\mathbb{C}[z] $-module if it is an $\mathcal{O}(\bar{\Omega})$-module,  Corollary \ref{fos} indicates that the $\mathcal{O}(\bar{\Omega})$-module action is more extensive than $\mathbb{C}[z] $-module action since a nonconstant polynomial is not invertible in the polynomial algebra $\mathbb{C}[z].$ Nevertheless, the following proposition shows that the above two notions of weighted Bergman submodules (and quotient submodules) coincide on bounded symmetric domains.

 \begin{prop}\label{CO} Suppose $M$ is a closed subspace of $A^2(dv_\gamma),$ then the following hold:
\begin{enumerate}
\item  $M$ is a $\mathbb{C}[z]$-module if and only if $M$ is an $\mathcal{O}(\bar{\Omega})$-module.
\item Moreover,  $M^\bot$ is a quotient submodule over the algebra $\mathbb{C}[z]$ if and only if $M^\bot$ is a quotient submodule over the algebra  $\mathcal{O}(\bar{\Omega}).$
\end{enumerate}
\end{prop}
\begin{proof}
(1) The sufficiency part is clear,  it is enough to prove the necessity part. Suppose that $M$ is a $\mathbb{C}[z]$-module of $A^2(dv_\gamma) $ and  $f\in \mathcal{O}(\bar{\Omega}),$ we have to prove that  $f\cdot g=f\vert_{\Omega}g \in M$ for every $g\in M.$ Since $f\in \mathcal{O}(\bar{\Omega}),$ there exists an open set $D$ which contains the compact set $\bar{\Omega}$ such that $f$ is holomorphic on $D.$ Since the open set $\Omega$ is the unit ball in the spectral norm, and the topology induced by the spectral norm coincides with the topology induced by the Euclidean norm,  it follows that there exists a $\delta>0$ satisfying (\ref{exto}). 
Note that $(1+\delta)\Omega$ is a circular domain containing the origin. Thus $f\in \mathcal{O}(D)\subset  \mathcal{O}((1+\delta)\Omega)$ has a homogeneous expansion \begin{equation}\label{hop} f(z)=\sum_{i=0}^\infty f_i(z)\end{equation} on the domain $(1+\delta)\Omega $ where each $f_i$ is $i$-homogeneous holomorphic polynomial, which converges compactly on $(1+\delta)\Omega.$ Hence the expansion (\ref{hop}) converges uniformly on $\Omega$ since $\Omega\subsetneq\bar{\Omega}\subsetneq (1+\delta)\Omega.$ Therefore, we have \begin{equation}\label{fglm} f\cdot g=\lim_i \sum_{j=0}^if_j\cdot g= \lim_i \sum_{j=0}^if_j g \in M,\end{equation} as $M $ is closed in $A^2(dv_\gamma).$  

(2) It is also enough to prove the necessity part.  Since $M^\bot$   is a quotient submodule over the algebra $\mathbb{C}[z]$ of $A^2(dv_\gamma),$  it follows that  $M$ is a $\mathbb{C}[z]$-submodule of $A^2(dv_\gamma).$ It follows from (1) that $M$ is an $ \mathcal{O}(\bar{\Omega})$-submodule of $A^2(dv_\gamma).$ Hence $M^\bot$ is a quotient submodule over the algebra $ \mathcal{O}(\bar{\Omega}).$
\end{proof}

\begin{rem}\label{Steh}
 For a moment, we will reinterpret Proposition  \ref{CO} from the point of view of the category.  Denote by $\mathbb{C}[z]$-Bmod the category of weighted Bergman modules on $\Omega,$ whose objects are submodules and quotient submodules of weighted Bergman module $A^2(dv_{\gamma})$ over the algebra $\mathbb{C}[z]$ for some $\gamma>-1$ and morphisms are continuous $\mathbb{C}[z]$-module homomorphisms. Similarly, $\mathcal{O}(\bar{\Omega})$-Bmod is denoted by the category of weighted Bergman modules on $\Omega,$  whose objects are submodules and quotient submodules of weighted Bergman module $A^2(dv_{\gamma})$ over the algebra $\mathcal{O}(\bar{\Omega})$ for some $\gamma>-1$ and morphisms are continuous $\mathcal{O}(\bar{\Omega})$-module homomorphisms.  Let $M_1\overset{\varphi}{\rightarrow} M_2$ be a morphism in $\mathbb{C}[z]$-Bmod. Then  Proposition \ref{CO} implies that $M_1, M_2$ are two objects in  $\mathcal{O}(\bar{\Omega})$-Bmod. Moreover, since $\varphi$ is continuous, combining this with  (\ref{fglm}) follows that $M_1\overset{\varphi}{\rightarrow} M_2$ is actually a morphism in $\mathcal{O}(\bar{\Omega})$-Bmod. This introduces a functor $\iota$ satisfying
 \begin{equation}\begin{split}\label{}
\iota:\mathbb{C}[z]\text{-Bmod}&\rightarrow\mathcal{O}(\bar{\Omega})\text{-Bmod},\notag\\
M_1\overset{\varphi}{\rightarrow} M_2&\overset{\iota}{\longmapsto} M_1\overset{\varphi}{\rightarrow} M_2.\\
\end{split}\end{equation}
Clearly $\mathcal{O}(\bar{\Omega})$-Bmod is a subcategory of $\mathbb{C}[z]$-Bmod in general, whose injection functor is denoted by $\tau.$ Then the direct check shows that $\iota\tau=\text{Id}$ and $\tau\iota=\text{Id}.$ The above shows that if $\Omega$  is an irreducible bounded symmetric domain then  $$\mathbb{C}[z]\text{-Bmod}=\mathcal{O}(\bar{\Omega})\text{-Bmod}$$ and $\iota$ is the identity functor (or an isomorphism).  
\end{rem}
Note that the operator problems we are concerned with are independent of the choice of the two above module categories. Thus in the sequel,  we will consider the $p$-essential normality problems in the context of $\mathcal{O}(\bar{\Omega})$-modules rather than  $\mathbb{C}[z]$-modules. Suppose $V$ is an analytic variety in $\Omega,$ as we mentioned before, which can be globally defined, namely, there exist finitely many holomorphic functions $f_1,\cdots,f_m\in\mathcal{O}(\Omega)$ satisfying $$V=\{z\in \Omega: f_1(z)=\cdots=f_m(z)=0\}.$$
We define $M_V$ by  \begin{equation}\label{mvf} M_V:=\{f\in  A^2(dv_\gamma):f(z)=0, \forall z\in V\},\end{equation}
Since evaluation functionals  are continuous on $ A^2(dv_\gamma),$ it follows that $M_V$ is an   $\mathcal{O}(\bar{\Omega})$-submodule of  $ A^2(dv_\gamma).$
 If one assumes further all $f_1,\cdots,f_m\in\mathcal{O}(\bar{\Omega}),$ then $V$ can be naturally extended to an open neighborhood of $\bar{\Omega},$ this is the most considered case in the generalized Geometric Arveson-Douglas  conjecture \cite{DoWy17, GuW20}. For $I$ is a  (prime) proper ideal in $\mathbb{C}[z],$ its common zeros $Z(I)$ is an (irreducible) algebraic variety. Let $V_I=Z(I)\cap \Omega,$ it is clear that the closure $\text{Cl}[I]\subset M_{V_I}$ is a submodule. We denote $\mathcal{Q}_I$ by its quotient submodule.  In some mild conditions, it can be showed \cite{DTY16} that  $\text{Cl}[I]= M_{V_I}.$  However, we do not know the general case.

Let $\phi: \Omega_1\rightarrow \Omega_2 $ be a holomorphic map between two bounded symmetric domains. The   pullback of functions by the holomorphic map $\phi$  is given by $$\phi^*:\mathcal{O}(\Omega_2)\rightarrow\mathcal{O}(\Omega_1),\quad f\mapsto \phi^*(f)= f\circ\phi,$$
which is a complex algebraic (module) homomorphism, especially a linear operator.
The measure $dv_\gamma$ can also be pulled back by  the holomorphic map $\phi,$ which is defined by $$\phi^*(dv_\gamma)(w):=c_\gamma \Delta(\phi(w),\phi(w))^{\gamma} dv(\phi(w)).$$ When  $\phi: \Omega_1\rightarrow \Omega_2 $ is biholomorphic,
then   $\phi^*:A^2(\Omega_2,dv_\gamma)\rightarrow A^2(\Omega_1,\phi^*(dv_\gamma))$ is a unitary operator. Not to be confused with the concept of adjoint of an operator.  Let $M$ be an   $\mathcal{O}(\bar{\Omega}_1)$-submodule of  $ A^2(\Omega_1,dv_\gamma)$  which is determined by an  analytic variety. 
In this case, the pullback $\phi^*(M)=\{\phi^*(f):f\in M\}$ is an   $\mathcal{O}(\bar{\Omega}_2)$-submodule of  $ A^2(\Omega_2,\phi^*(dv_\gamma)),$   which is determined by the  analytic variety $\phi^{-1}(V)=\{\phi^{-1}(z):z\in V\} .$ 
Moreover, $\phi^*(M)$ is unitarily equivalent to $M.$ Hence the pullback $\phi^*(M^\bot)$ is a quotient submodule and  unitarily equivalent to $M^\bot,$ which implies that \begin{equation}\label{pmp}\phi^* P_{M^\bot}=P_{\phi^*(M^\bot)}\phi^*.\end{equation}
\begin{prop}\label{MS} Suppose that $\phi: ({\Omega}_1,w)\rightarrow ({\Omega}_2,z) $ is a biholomorphism with $z=\phi(w),$ and let $M$ be a submodule of $A^2(\Omega_2,dv_\gamma)$  determined by an analytic variety, then the following hold:
\begin{enumerate}
\item $M_{z_i}^{M}=(\phi^*)^{-1} M_{\phi_i(w)}^{\phi^*(M)}\phi^*.$
\item $S_{z_i}^{M^\bot}=(\phi^*)^{-1} S_{\phi_i(w)}^{\phi^*(M^\bot)}\phi^{*}.$
\end{enumerate}
\end{prop}
\begin{proof}
(1) Let $f\in M.$ Since $\phi^*$ is an  algebraic homomorphism and $\phi^* (z_i)=w_i,$ it follows that 
  $$\phi^* (z_if)=\phi^* (z_i)\phi^* (f)=\phi_i(w)\phi^* (f)\in \phi^*(M),$$ which implies   $$\phi^* M_{z_i}=M_{\phi_i(w)}\phi^*,$$ i.e., $M_{z_i}=(\phi^*)^{-1} M_{\phi_i(w)}\phi^*,$  for  $i=1,\cdots,n.$

(2) Similarly, combined with (\ref{pmp}), it follows that
\begin{equation}\begin{split}\label{}
\phi^* S_{z_i}^{M^\bot}f&=\phi^* P_{M^\bot}M_{z_i}f\notag\\
&=\phi^*P_{M^\bot}(\phi^*)^{-1}\phi^*M_{z_i}f\\
&=\phi^* P_{M^\bot}(\phi^*)^{-1}\phi^*(z_if)\\
&=\phi^* P_{M^\bot}(\phi^*)^{-1}\phi_i(w)\phi^* f\\
&=P_{\phi^*(M^\bot)}\phi_i(w)\phi^* f
\end{split}\end{equation}
for every $ f\in M^\bot$ and  $i=1,\cdots,n.$ Thus $S_{z_i}^{M^\bot}=(\phi^*)^{-1} S_{\phi_i(w)}^{\phi^*(M^\bot)}\phi^{*},$ which completes the proof.
\end{proof}

We are now in a position to complete the proof of  Theorem 2, we first prove a special case.
\begin{lem}\label{kk} Let $M$ be a weighted Bergman submodule determined by an analytic variety in $\Omega.$ Suppose that $\phi: (\Omega,w)\rightarrow (\Omega,z) $ is a linear automorphism  that is $\phi\in K$, then the $\mathbb{C}[w]$-module  $(\phi^*(M^\bot),S_w)$ is $p$-essentially normal if and only if  the $\mathbb{C}[z]$-module  $(M^\bot,S_z)$  is $p$-essentially normal.
\end{lem}
\begin{proof} Since $\phi: (\Omega,w)\rightarrow (\Omega,z) $ is a linear automorphism, without loss of generality, we can assume   $z_i=\phi_i(w)=\sum_{j=1}^n a_i^jw_j,$  where $(a_i^j)$ is an invertible matrix. 
It follows from Proposition \ref{MS}\hspace{0.1em}(2) that
\begin{equation}\begin{split}\label{}
\phi^{*}[S_{z_i}^{M^\bot},(S_{z_j}^{M^\bot})^*](\phi^*)^{-1}&=[S_{\phi_i(w)}^{\phi^*(M^\bot)},(S_{\phi_j(w)}^{\phi^*(M^\bot)})^*]\notag\\
&=[S_{\sum_{l=1}^na_i^lw_l }^{\phi^*(M^\bot)},(S_{\sum_{m=1}^na_j^mw_m}^{\phi^*(M^\bot)})^*] \\
&= \sum_{l=1}^n \sum_{m=1}^na_i^l\bar{a}_j^m[S_{w_l }^{\phi^*(M^\bot)},(S_{w_m}^{\phi^*(M^\bot)})^*],
\end{split}\end{equation}
where $i,j=1,\cdots,n.$
  Hence $[S_{z_i}^{M^\bot},(S_{z_j}^{M^\bot})^*]\in\mathcal{L}^{p}(M^\bot)$ if $[S_{w_i}^{\phi^*(M^\bot)},(S_{w_j}^{\phi^*(M^\bot)})^*]\in\mathcal{L}^{p}(\phi^\ast(M^\bot)).$
 Thus $(M^\bot,S_z)$  is $p$-essentially normal if $(\phi^*(M^\bot),S_w)$ is $p$-essentially normal. Since the matrix $(a_i^j)$ is invertible, we denote its inverse by $(b_i^j),$ it follows that $w_i=\sum_{j=1}^nb_i^jz_j.$ The same argument shows that $(\phi^*(M^\bot),S_w)$ is $p$-essentially normal if $(M^\bot,S_z)$  is $p$-essentially normal.
\end{proof}

{\noindent{\bf{Proof of  Theorem 2.}}
(1) Without loss of generality, we suppose  $\phi(\bm{0})=z_0.$   It follows from the proof of Lemma \ref{bihg} that the  automorphism $\phi$ admits the factorization $\phi=g_{z_0}\circ k$ where $k\in K,$ which means that $\phi$ admits the following commutative diagram
\begin{displaymath}\xymatrix{\Omega \ar@{->}[dr]_{\phi} \ar[r]^k &\Omega \ar[d]^{g_{z_0}}\\ &\Omega}
\end{displaymath}

By Lemma \ref{kk}, the $p$-essential normality of quotient weighted  Bergman submodules is invariant under the linear automorphism $k$, so it suffices to prove the conclusion in the case of $\phi=g_{z_0}.$ On the other hand, by \cite[Section 1.5]{Upm96}, if the quasi-inverse of the pair $(z,\xi)\in\mathbb{C}^n\times\mathbb{C}^n$ exists, then its quasi-inverse can be written in the rational form of $$z^{\xi}=\frac{p(z,\xi)}{\Delta(z,\xi)},$$ where $p(z,\xi)$ is a $\mathbb{C}^n$-valued polynomial in $z,\bar{\xi},$ $ \Delta(z,\xi)$ is the Jordan triple determinant and has no common factors with $p(z,\xi).$ Moreover,  by \cite[Section 7.5]{Loo77}, the quasi-inverse of the pair $(z,\xi)$ exists whenever the spectrum norm satisfying $\Vert z\Vert\cdot\Vert \xi\Vert<1,$ and $ \Delta(z,\xi)\neq 0$ whenever the quasi-inverse of the pair $(z,\xi)$ exists. Therefore, for every $z_0\in \Omega=\{\Vert z\Vert<1\},$  the biholomorphism $g_{z_0}\in\text{Aut}(\Omega)$ is in fact holomorphic on the irreducible bounded symmetric domain $\Omega_{z_0}=\{\Vert z\Vert<\frac{1}{\Vert z_0\Vert}\}\supset\bar{\Omega}$ and $g_{z_0}:\bar{\Omega}\rightarrow\bar{\Omega}$ is biholomorhphic by Lemma \ref{bihg}. Thus each component function of $g_{z_0}$ can be uniquely extended to a holomorphic function in $\mathcal{O}(\Omega_{z_0}),$ it implies that each component function of $g_{z_0}$ is a rational function and belongs to $\mathcal{O}(\bar{\Omega}).$ So, we can write $$g_{z_0}=(\frac{p_1}{q_1},\cdots,\frac{p_n}{q_n}),$$ where  rational functions $\frac{p_i}{q_i}$ belong to $\mathcal{O}(\bar{\Omega})$ and $q_i$ have  no zeros in $\bar{\Omega}.$   Applying Proposition \ref{MS}\hspace{0.1em}(2) to the case of $\phi=g_{z_0},$ and combining this with Lemma \ref{age} and Corollary \ref{fos}, we obtain
\begin{equation}\begin{split}\label{}
\phi^{*}[S_{z_i}^{M^\bot},(S_{z_j}^{M^\bot})^*](\phi^*)^{-1} &=[S_{\frac{p_i}{q_i}}^{\phi^*(M^\bot)},(S_{\frac{p_j}{q_j}}^{\phi^*(M^\bot)})^*]\notag\\
&= [S_{p_i}^{\phi^*(M^\bot)}({S_{q_i}^{\phi^*(M^\bot)}})^{-1},(S_{p_j}^{\phi^*(M^\bot)})^*((S_{q_j}^{\phi^*(M^\bot)})^*)^{-1}]\\
&= S_{p_i}^{\phi^*(M^\bot)}[{(S_{q_i}^{\phi^*(M^\bot)})}^{-1},(S_{p_j}^{\phi^*(M^\bot)})^*((S_{q_j}^{\phi^*(M^\bot)})^*)^{-1}]\\
&\quad+[S_{p_i}^{\phi^*(M^\bot)},(S_{p_j}^{\phi^*(M^\bot)})^*((S_{q_j}^{\phi^*(M^\bot)})^*)^{-1}](S_{q_i}^{\phi^*(M^\bot)})^{-1}\\
&= U_1+U_2+U_3+U_4,
\end{split}\end{equation}
where
\begin{equation}\begin{split}\label{}
&U_1=S_{p_i}^{\phi^*(M^\bot)}S_{p_j}^{\phi^*(M^\bot)}(S_{q_i}^{\phi^*(M^\bot)})^{-1}(S_{q_j}^{\phi^*(M^\bot)})^{-1} [{S_{q_i}^{\phi^*(M^\bot)}},(S_{q_j}^{\phi^*(M^\bot)})^*]\notag\\
&\quad\quad \cdot((S_{q_j}^{\phi^*(M^\bot)})^*)^{-1}(S_{q_i}^{\phi^*(M^\bot)})^{-1},\\
& U_2=-S_{p_i}^{\phi^*(M^\bot)}(S_{q_i}^{\phi^*(M^\bot)})^{-1}[S_{q_i}^{\phi^*(M^\bot)},(S_{p_j}^{\phi^*(M^\bot)})^*](S_{q_i}^{\phi^*(M^\bot)})^{-1}((S_{q_j}^{\phi^*(M^\bot)})^*)^{-1},\\
&U_3=-(S_{p_j}^{\phi^*(M^\bot)})^*((S_{q_j}^{\phi^*(M^\bot)})^*)^{-1}[S_{p_i}^{\phi^*(M^\bot)},(S_{q_j}^{\phi^*(M^\bot)})^*]((S_{q_j}^{\phi^*(M^\bot)})^*)^{-1}(S_{q_i}^{\phi^*(M^\bot)})^{-1},\\
&U_4=[S_{p_i}^{\phi^*(M^\bot)},(S_{p_j}^{\phi^*(M^\bot)})^*]((S_{q_j}^{\phi^*(M^\bot)})^*)^{-1}(S_{q_i}^{\phi^*(M^\bot)})^{-1},
\end{split}\end{equation}
for all $i,j=1,\cdots,n.$
 Since $\mathcal{L}^{p}(\phi^\ast(M^\bot))$ is an operator ideal, it follows from Proposition \ref{Sij} that $[S_{z_i}^{M^\bot},(S_{z_j}^{M^\bot})^*]\in \mathcal{L}^{p}(M^\bot)$ if $[S_{w_i}^{\phi^*(M^\bot)},(S_{w_j}^{\phi^*(M^\bot)})^*]\in\mathcal{L}^{p}(\phi^\ast(M^\bot)).$  And hence $(M^\bot,S_z)$  is $p$-essentially normal if $(\phi^*(M^\bot),S_w)$ is $p$-essentially normal. Similarly, the same argument for the  automorphism $\phi^{-1}: (\Omega,z)\rightarrow (\Omega,w) $  follows that $(\phi^*(M^\bot),S_w)$ is $p$-essentially normal if $(M^\bot,S_z)$  is $p$-essentially normal.  

(2) Since  $\phi^*:A^2(\Omega_2,dv_\gamma)\rightarrow A^2(\Omega_1,\phi^*(dv_\gamma))$ is a unitary operator, it follows from Proposition \ref{MS}\hspace{0.1em}(2) that
$$[S_{z_i}^{M^\bot},(S_{z_j}^{M^\bot})^*]= (\phi^*)^{-1} [S_{\phi_i(w)}^{\phi^*(M^\bot)}, (S_{\phi_j(w)}^{\phi^*(M^\bot)})^\ast]\phi^{*}$$ for all $i,j=1,\cdots,n.$ Moreover, $(\phi^*(M^\bot),S_w)$ is $p$-essentially normal if and only if $(M^\bot,S_z)$  is $p$-essentially normal from Theorem 2\hspace{0.1em}(1), and then the desired result follows.
\qed

Theorem 2 provides a result that the $p$-essential normality is independent of the choice of the holomorphic coordinate charts.  
Theorem 2 can be regarded as a generalization of \cite[Corollary 3.2]{Sal89} for the $p$-essential normality of quotient submodules, since proper holomorphic self-maps on irreducible bounded symmetric domains with higher dimensions must be automorphisms.  Moreover, it can be checked that quotient submodules $(\phi^*(M^\bot),S_w)$  and $(\phi^*(M^\bot),S_\phi)$ over the same polynomial algebra $\mathbb{C}[w]$ are not $2p$-isomorphic (in the sense of Arveson \cite{Arv07}) in general. 

   \begin{exm} (1) Let $h$ be an irreducible homogeneous polynomial.  Suppose that the variety $V=Z(h)\cap \mathbb{B}$ is smooth away from the origin, where $Z(h)$ is the set of zeros of $h.$ It follows from \cite{DTY16,GuW20} that the quotient Bergman submodule $M_V^\bot$ on $\mathbb{B}^n$ is $p$-essentially normal for all $p>n$.  Suppose that $\phi: (\mathbb{B}^n,w)\rightarrow (\mathbb{B}^n,z) $ is an automorphism with $z=\phi(w)$, then Theorem 2 implies that  
$[S_{\phi_i},S_{\phi_j}^*]\in\mathcal{L}^p(M_V^\bot)$ for all $p>n$ and  $i,j=1,\cdots,n.$

(2) Let $\tilde{V}$ be an analytic variety in an open neighborhood of $\bar{\mathbb{B}}^n$ such that  $\tilde{V}$ intersects $\partial\mathbb{B}^n$ transversely and has no singularities on $\partial\mathbb{B}^n.$  Denote $V=\tilde{V}\cap\mathbb{B}^n$. It follows from \cite[Theorem 4.3]{GuW20} that the quotient Bergman submodule $(M_V^\bot,S_w)$ is  $p$-essentially normal for all $p>2\hspace{0.1em}\text{dim}_\mathbb{C}\hspace{0.1em} V.$ Suppose that $\phi: (\mathbb{B}^n,w)\rightarrow (\mathbb{B}^n,z) $ is an automorphism with $z=\phi(w)$, then Theorem 2 implies that  
$[S_{\phi_i},S_{\phi_j}^*]\in\mathcal{L}^p(M_V^\bot)$ for all $p>2\hspace{0.1em}\text{dim}_\mathbb{C}\hspace{0.1em} V$ and  $i,j=1,\cdots,n.$

  \end{exm}
  \begin{rem}\label{szfg} By the argument of the Taylor functional calculus as we used in the proof of Theorem \ref{th1}, we have the following further result.
Let  $(\Omega,z)$  be a bounded symmetric domain in $\mathbb{C}^n$ and  $(M^\bot, S_z)$ be a quotient Bergman submodule determined by an analytic variety. If the Taylor spectrum  $Sp(S_z)\subset \bar{\Omega},$ then the following are equivalent:
\begin{enumerate}
\item $[S_{z_i},S_{z_j}^*]\in \mathcal{L}^p(M^\bot)$ for all $i,j=1,\cdots,n;$
\item $[S_{f},S_{g}^*]\in\mathcal{L}^p(M^\bot)$ for all $f,g\in \mathcal{O}(\bar{\Omega}).$
\end{enumerate}
The condition $Sp(S)\subset \bar{\Omega}$ is satisfied in many known cases where the Taylor spectrum  $Sp(S_z)$ is calculated. A special case is that $Sp(S_z)=\bar{\Omega}$ when $M=\bm{0}$ by \cite[Theorem 4.1]{AZ03}, thus the weighted Bergman module $(A^2(dv_\gamma),T_z)$ is $p$-essentially normal if and only if  $[T_{f},T_{g}^*]\in\mathcal{L}^p(A^2(dv_\gamma))$ for all $f,g\in \mathcal{O}(\bar{\Omega}).$  This gives another intrinsic characterization of the $p$-essential normality of weighted Bergman modules. 
\end{rem}

 \section{The case on the Wallach set}
  In the preceding section, we consider the $p$-essential normality of weighted Bergman modules on bounded symmetric domains. However, as mentioned at the end of Section 2,  the weighted Bergman spaces are special cases of the reproducing kernel Hilbert holomorphic function spaces corresponding to the continuous part of the Wallach set. In this section, we will consider the biholomorphic invariance of $p$-essential normality of  Hilbert modules on the more general holomorphic function space determined by the Wallach set.


 Although we still consider the operator problems in the $\mathcal{O}(\bar{\Omega})$-module category, the situation is different in general. When $ \lambda>N-1,$ it is trivial to equip the weighted Bergman spaces $H^2_\lambda(\Omega)$ with module actions by the multiplication of functions over algebras $\mathcal{P}(Z)$ and $\mathcal{O}(\bar{\Omega})$ since those inner products are induced by the integration with suitable probability measures. However,  for general $\lambda\in W_{\Omega,c},$ we require additional effort to equip $H^2_\lambda(\Omega)$ with an  $\mathcal{O}(\bar{\Omega})$-module action by the multiplication of functions when it is a Hilbert module over the algebra $\mathcal{P}(Z),$ since the $K$-invariant inner product on $H^2_\lambda(\Omega)$  can not be induced by the integration with a probability measure supported on $\bar{\Omega}.$  Moreover,   for  an arbitrary  function   $f\in\mathcal{O}(\bar{\Omega}),$  it is nontrivial to determine whether $f\in H^2_\lambda(\Omega)$  or not for general $ \lambda<N-1 .$

We first show that the holomorphic function space $H^2_\lambda(\Omega)$ on  $\Omega$ determined by the discrete part $W_{\Omega,d}$ of the Wallach set is not a Hilbert module over the polynomial algebra $\mathcal{P}(Z)=\mathbb{C}[z].$ However, for every continuous point $\lambda\in W_{\Omega,c},$    the function space $(H^2_\lambda(\Omega), T_z)$ is always a Hilbert module over $\mathcal{P}(Z)$ by the coordinate function multipliers, which is a restatement of \cite[Theorem 4.1]{AZ03}. We formulate it in the following lemma.
   \begin{lem}\label{Wal} Let $\lambda\in W_{\Omega},$  then $(H^2_\lambda(\Omega),T_z)$ is  a Hilbert module over $\mathcal{P}(Z)$ with the  coordinate function multipliers if and only if $\lambda\in W_{\Omega,c}.$ Moreover, the Taylor spectrum  $Sp(T_z)=\bar{\Omega}$ whenever  $\lambda\in W_{\Omega,c}.$
   \end{lem}
\begin{proof} By \cite[Theorem 4.1]{AZ03}  the function space $(H^2_\lambda(\Omega), T_z)$  with the coordinate function multipliers is always a Hilbert module over $\mathcal{P}(Z)$  and the Taylor spectrum  $Sp(T_z)=\bar{\Omega}$ whenever  $\lambda\in W_{\Omega,c}.$  Thus it is sufficient to prove that $H^2_\lambda(\Omega)$ is  not a Hilbert module over $\mathcal{P}(Z)$ by the  coordinate function multipliers  if $\lambda\in W_{\Omega,d}.$
Suppose it was false. Then we could find $\lambda_0 =\frac{a}{2}(j_0-1)$ for  some $1\leq j_0\leq r$ such that $H^2_{\lambda_0}(\Omega)$ is a Hilbert module over $\mathcal{P}(Z)$ by the  coordinate function multipliers.  By the construction of $H^2_{\lambda_0}(\Omega),$ we know that $$H^2_{\lambda_0}(\Omega)=\sum_{\bm{m}\geq0,m_{j_0}=0} \mathcal{P}_{\bm{m}}(Z).$$ Since the constant function  $ \bm{1}\in H^2_{\lambda_0}(\Omega),$  it follows that  $$ \mathcal{P}(Z)\cdot \bm{1}=\mathcal{P}(Z)\subset H^2_{\lambda_0}(\Omega),$$  which contradicts $\mathcal{P}(Z)=\sum_{\bm{m}\geq0} \mathcal{P}_{\bm{m}}(Z),$ and the proof is finished.
\end{proof}

 In Section \ref{seon}, we define an $\mathcal{O}(\bar{\Omega})$-module action by the function multipliers on the weighted Bergman spaces $H^2_{N+\gamma}(\Omega)=A^2(dv_\gamma),\gamma>-1.$ Notice that $$N+\gamma>N-1>(r-1)\frac{a}{2}.$$

 Despite the aforementioned difficulties, we can still define an $\mathcal{O}(\bar{\Omega})$-module action by the function multipliers on $H^2_\lambda(\Omega)$ for all $\lambda\in W_{\Omega,c}=\{\lambda>(r-1)\frac{a}{2}\}.$ To cover the gap $\{ (r-1)\frac{a}{2} <\lambda\leq N-1\},$ we need the following embedding lemma between two function spaces.  To make its proof more precise, we introduce the following notations. Let $I_\Omega$ be the set of all integer partitions with length $r.$ Define
 \begin{equation}\begin{split}
 &I_\Omega(0)=\{\bm{0}\},\notag\\
 &I_\Omega(j)=\{(m_1,\cdots,m_j,0,\cdots,0):m_1\geq\cdots\geq m_j>0\},1\leq j\leq r-1,\\
 &I_\Omega(r)=\{(m_1,\cdots,m_r):m_1\geq\cdots\geq m_r>0\}.
 \end{split}\end{equation}
 It is obvious that \begin{equation}\label{}I_\Omega=\bigcup_{j=0}^r I_\Omega(j)\quad \text{and} \quad I_\Omega(i)\cap I_\Omega(j)=\emptyset\notag\\\end{equation} whenever $i\neq j.$
 \begin{lem}\label{imbe} Suppose $ (r-1)\frac{a}{2}< \lambda_1< \lambda_2,$ then the imbedding $$ i: H^2_{\lambda_1}(\Omega) \rightarrow H^2_{\lambda_2}(\Omega),\quad f\mapsto i(f)=f$$ is continuous. 
 \end{lem}
\begin{proof}
Let $f \in H^2_{\lambda}(\Omega)$ for $\lambda\in W_{\Omega,c}$ and $f=\sum_{\bm{m}\geq0} f_{\bm{m}} $ be its Peter-Schmid-Weyl decomposition.  Then $$\Vert f\Vert_\lambda=\sum_{\bm{m}\geq0} \Vert f_{\bm{m}}\Vert_\lambda.$$
Hence it suffices to prove that there exists a uniform positive constant $C$ satisfying
\begin{equation}\label{}
\Vert p_{\bm{m}}\Vert_{\lambda_2}\leq C\Vert p_{\bm{m}}\Vert_{\lambda_1}\notag\\
\end{equation}
for all $ p_{\bm{m}}\in \mathcal{P}_{\bm{m}}$ and $\bm{m}\geq0.$ Combining this with the formula (\ref{pqin}),  it suffices to prove that there exists a uniform positive constant $C$ satisfying \begin{equation}\label{laco}\frac{(\lambda_1)_{\bm{m}}}{(\lambda_2)_{\bm{m}}}\leq C\end{equation} for every  $\bm{m}\geq0.$
 By the definition of the multi-variable Pochhammer symbol, we obtain
  \begin{equation}\label{poc}
  \frac{(\lambda_1)_{\bm{m}}}{(\lambda_2)_{\bm{m}}}=\frac{\Gamma_\Omega(\lambda_1+\bm{m})}{\Gamma_\Omega(\lambda_2+\bm{m})}
                                        =\prod_{j=1}^r\frac{\Gamma(\lambda_1+m_j-(j-1)\frac{a}{2})}{\Gamma(\lambda_2+m_j-(j-1)\frac{a}{2})}.\\
 \end{equation}
   It is easy to see that \begin{equation}\label{p0}\frac{(\lambda_1)_{\bm{0}}}{(\lambda_2)_{\bm{0}}}=\prod_{j=1}^r\frac{\Gamma(\lambda_1-(j-1)\frac{a}{2})}{\Gamma(\lambda_2-(j-1)\frac{a}{2})}.\end{equation}
   By Stirling's formula,  there exists an uniform positive constant $C_j$ satisfying$$ \frac{\Gamma(\lambda_1+m_j-(j-1)\frac{a}{2})}{\Gamma(\lambda_2+m_j-(j-1)\frac{a}{2})}\leq C_j\frac{1}{m_j^{\lambda_2-\lambda_1}},$$ for $j=1,\cdots,r.$
   Then (\ref{poc}) implies that \begin{equation}\begin{split}\label{p1}
  \frac{(\lambda_1)_{\bm{m}}}{(\lambda_2)_{\bm{m}}}
                                        &\leq\prod_{j=1}^k C_j\prod_{j=k+1}^r \frac{\Gamma(\lambda_1-(j-1)\frac{a}{2})}{\Gamma(\lambda_2-(j-1)\frac{a}{2})}\prod_{j=1}^k\frac{1}{m_j^{\lambda_2-\lambda_1}}\\
                                         &\leq\prod_{j=1}^k C_j\prod_{j=k+1}^r \frac{\Gamma(\lambda_1-(j-1)\frac{a}{2})}{\Gamma(\lambda_2-(j-1)\frac{a}{2})}
     \end{split}\end{equation}
     for every $\bm{m}\in I_\Omega(k), 1\leq k\leq r-1.$  Similarly, \begin{equation}\begin{split}\label{p2}
  \frac{(\lambda_1)_{\bm{m}}}{(\lambda_2)_{\bm{m}}}
                                        & \leq \prod_{j=1}^r C_j \\
     \end{split}\end{equation}
 for every $\bm{m}\in I_\Omega(r). $  Thus (\ref{laco}) follows by (\ref{p0}), (\ref{p1}) and (\ref{p2}).  This completes the proof.
 \end{proof}
Lemma \ref{imbe} generalizes the classical fact that the Hardy space is a subspace of the Bergman space on the unit ball of rank one.  
\begin{lem}\label{Hoo} Let $\lambda\in W_{\Omega,c},$   then $ H^2_\lambda(\Omega)$ is an $\mathcal{O}(\bar{\Omega})$-module whose module actions are  the multiplication of functions. In particular, $\mathcal{O}(\bar{\Omega})\subset H^2_\lambda(\Omega)$ is dense.
\end{lem}
\begin{proof} The case of $\lambda>N-1$ is clear, since in this case every $ H^2_\lambda(\Omega)$ is a weighted Bergman space whose inner product is induced by a finite volume measures on $\Omega.$ Thus it suffices to consider the case of $\lambda\in \{ (r-1)\frac{a}{2} <\lambda\leq N-1\}.$ Since the Taylor spectrum $Sp(T_z)=\bar{\Omega}$ by Lemma \ref{Wal}, it follows that we can define an $\mathcal{O}(\bar{\Omega})$-module  action on $H^2_\lambda(\Omega)$ by Taylor functional calculus, namely   \begin{equation}\label{Tafh}f\cdot h:= f(T_z)h\end{equation} for every  $f\in\mathcal{O}(\bar{\Omega})$ and $h\in H^2_\lambda(\Omega).$
Since $f\in \mathcal{O}(\bar{\Omega}),$ there exists an open set $D$ which contains the compact set $\bar{\Omega}$ such that $f$ is holomorphic on $D.$ It follows from (\ref{hop}) that  there exists a $\delta>0$ such that  $$\Omega\subsetneq (1+\delta)\Omega\subsetneq (1+\delta)\bar{\Omega}\subsetneq D$$ and $f\in \mathcal{O}(D)\subset  \mathcal{O}((1+\delta)\Omega)$ has a homogeneous polynomial expansion \begin{equation}\label{} f(z)=\sum_{i=0}^\infty f_i(z)\notag\\\end{equation} on the domain $(1+\delta)\Omega $ where each $f_i$ is $i$-homogeneous holomorphic polynomial, which converges uniformly on $(1+\frac{\delta}{2})\Omega.$
Thus $f(T_z)=\lim_i \sum_{j=0}^i f_j(T_z)$ in the strong operator topology  by the Taylor functional calculus, which yields  \begin{equation}\label{eva}f\cdot h= f(T_z)h=\lim_i \sum_{j=0}^if_j(T_z)h=\lim_i \sum_{j=0}^if_jh\end{equation}
for every  $h\in H^2_\lambda(\Omega).$ On the other hand, it follows from Lemma \ref{imbe}  that the imbedding $H^2_\lambda(\Omega)\subset H^2_N(\Omega)=A^2(dv)$ is continuous  because $(r-1)\frac{a}{2} <\lambda< N.$  Then $\sum_{j=0}^if_jh$ is convergent to $f\cdot h$ in $H^2_N(\Omega).$ Since  $f=\lim_{i} \sum_{j=0}^i f_j$ converges uniformly on $(1+\frac{\delta}{2})\Omega,$ it follows that $\sum_{j=0}^if_jh$ converges to $fh$ in $H^2_N(\Omega)$ by the integral formula of the norm with respect to the measure $dv,$ and hence \begin{equation}\label{evaa}f\cdot h=\lim_i \sum_{j=0}^if_jh=fh.\end{equation} This implies that we can define an $\mathcal{O}(\bar{\Omega})$-module action by the multiplication of functions, which coincides with the $\mathcal{O}(\bar{\Omega})$-module action  (\ref{Tafh}) defined by the Taylor functional calculus.
\end{proof}
\begin{rem} Actually  (\ref{evaa}) can be also obtained directly from (\ref{eva}) by the continuity of the evaluation functional. On the other hand, we know from \cite[Theorem 3.3]{Arv98} that  $\mathcal{A}(\Omega)=C(\bar{\Omega})\cap\mathcal{O}(\Omega)$  is not always contained in $H^2_\lambda(\Omega),\lambda\in W_{\Omega,c}$ in general.
\end{rem}
As a consequence, we can generalize Proposition \ref{CO} to the following form.
\begin{cor}\label{COl} Suppose  $M$ is a closed set of $H_\lambda^2(\Omega)$ for some  $\lambda\in W_{\Omega,c},$ then the following hold:
\begin{enumerate}
\item  $M$ is a $\mathbb{C}[z]$-module if and only if $M$ is an $\mathcal{O}(\bar{\Omega})$-module.
\item Moreover,  $M^\bot$ is a quotient submodule over the algebra $\mathbb{C}[z]$ if and only if $M^\bot$ is a quotient submodule over the algebra  $\mathcal{O}(\bar{\Omega}).$
\end{enumerate}
\end{cor}

Thus all concepts we defined in Section \ref{seon} for the weighted Bergman modules can be generalized to the current situation of the Hilbert modules $H_\lambda^2(\Omega)$ for $(r-1)\frac{a}{2}<\lambda\leq N-1,$ except the pullback of measures. It is an interesting question to construct integral formulas with suitable probability measures $dv_\lambda$  for the $K$-invariant inner products associated with the Wallach points; we refer the reader to \cite{AZ03, Upm96}. It is known that the $K$-invariant inner product on $H_\lambda^2(\Omega)$ for $\lambda\in W_{\Omega,c}$  can be realized as the integration with a unique probability measure $dv_\lambda$ supported inside $\bar{\Omega}$ if and only if $$\lambda\in \{\lambda_j=N-1-(j-1)\frac{a}{2},j=1,\cdots,r\}\cup\{\lambda>N-1\}.$$ 
Hence in the spacial case of $\lambda=\lambda_j$ we can still pull back  the measures $dv_{\lambda_j},j=1\cdots,r.$
However, in general, we will directly pull back the inner product rather than the measure as follows.
Let $\phi: \Omega\rightarrow \Omega $ be an automorphism. The pullback of functions  is given by $$\phi^*:\mathcal{O}(\Omega)\rightarrow\mathcal{O}(\Omega),\quad f\mapsto \phi^*(f)= f\circ\phi,$$
which is a complex algebraic (module) homomorphism and especially a linear operator.  For $H_\lambda^2(\Omega), \lambda\in W_{\Omega,c},$ its pullback under $\phi$ is denoted by $\phi^*(H_\lambda^2(\Omega)),$ we equip with it an inner product  by
 \begin{equation}\label{puin}\langle \phi^*(f),\phi^*(g)\rangle_{\lambda,\phi^*}:=\langle f,g\rangle_{\lambda}\end{equation} whenever $f,g\in H_\lambda^2(\Omega),$ which is called the pullback of the inner product. By direct check, we deduce that the inner product (\ref{puin}) is $K$-invariant and $\phi^*(H_\lambda^2(\Omega))$ is completed in this inner product, which coincides with the integration with respect to the pull backed measure whenever the inner product of $H_\lambda^2(\Omega)$  can be realized as the integration with respect to a  probability measure. Thus $\phi^*:H_\lambda^2(\Omega)\rightarrow\phi^*(H_\lambda^2(\Omega))$ is a unitary operator. Furthermore, if $M$ is an $\mathcal{O}(\bar{\Omega})$-submodule in $H_\lambda^2(\Omega)$ determined by an  analytic variety, then its pull back $\phi^*(M)$ is also an $\mathcal{O}(\bar{\Omega})$-submodule in $\phi^*(H_\lambda^2(\Omega))$ determined by  the pullback of the  analytic variety. The similar is also true for the quotient submodule.

 Now we list the result of the biholomorphic invariance of $p$-essential normality of quotient  Hilbert submodules in  $H_\lambda^2(\Omega)$ for $\lambda\in W_{\Omega,c},$  the core case of weighted Bergman modules of $\lambda>N-1$ is our main  Theorem 2 above. Now, with the preceding effort, the case of $(r-1)\frac{a}{2}<\lambda \leq N-1$
can be proved by the similar method to the case of weighted Bergman modules and we left this proof to the reader.

\begin{thm}\label{hwaa} Let $M$ be a  submodule in $H_\lambda^2(\Omega),\lambda\in W_{\Omega,c}$ determined by an  analytic variety.  Suppose that $\phi: (\Omega,w)\rightarrow (\Omega,z) $ is an arbitrary automorphism with $z=\phi(w)$, then the following hold:
 \begin{enumerate}
\item  The $\mathbb{C}[w]$-module $(\phi^*(M^\bot),S_w)$ is $p$-essentially normal if and only if the $\mathbb{C}[w]$-module  $(M^\bot,S_z)$  is $p$-essentially normal.
\item    The following two statements are equivalent.
\begin{enumerate}
\item $[S_{w_i},S_{w_j}^*]\in \mathcal{L}^p(M^\bot)$ for all $i,j=1,\cdots,n.$
\item $[S_{\phi_i},S_{\phi_j}^*]\in\mathcal{L}^p(M^\bot)$ for all $i,j=1,\cdots,n.$
\end{enumerate}
\end{enumerate}
\end{thm}

   \begin{rem}\label{beublaa}
   (1)  Since the unitary coordinate transformations correspond to the isometrically algebraic isomorphisms of the universal operator algebras \cite[Theorem 8.2]{DRS11}, Theorem \ref{hwaa}\hspace{0.1em}(1) can be viewed as an extension of \cite[Theorem 2.1]{KS12} for various Hilbert modules on bounded symmetric domains; however, the method we used here is different from theirs.

 (2) Let $\theta$ be a nonconstant singular inner function on $\mathbb{D}=\mathbb{B}^1$ and $M_\theta$ be the multiplication operator with symbol $\theta$ on Hardy space $H^2(\mathbb{D}).$ Clearly, $M_\theta$ coincides with  the Toeplitz operator $T_\theta.$  By the compactness characterization of  semi-commutators of Toeplitz operators, we conclude that
$[M_\theta,M_\theta^\ast]=T_\theta T_{\bar\theta}-T_{\theta\bar{\theta}}$ is not compact on  $H^2(\mathbb{D}).$  Thus the Hilbert module $(H^2(\mathbb{D}), M_\theta)$ is not $\infty$-essentially normal. On the other hand,  we can equip $H^2(\mathbb{D})$ with a natural  $H^\infty(\mathbb{D})$-module structure by the multiplication of functions, where $H^\infty(\mathbb{D})$ is the set of bounded holomorphic functions on $\mathbb{D}.$ Clearly, we have $\mathbb{C}[z]\subset\mathcal{O}(\bar{\mathbb{D}})\subset H^\infty(\mathbb{D}).$  The above shows that the larger function algebra $H^\infty(\mathbb{D})$ does not preserve the  $\infty$-essential normality of  $(H^2(\mathbb{D}), M_z).$

 (3) 
 When the  Taylor spectrum of the  compression $n$-tuple $S_z$ on $M_V^\bot$ satisfies  $Sp(S_z)\subset\bar{\Omega}.$ One can prove Theorem \ref{hwaa} by the alternative operator theoretic method based on Theorem 1 (taking $c=1,\bm{d}=\bm{0}$). 
  However, in the general case, as we mentioned before, the primary obstruction in this operator theoretic method is the  accurate estimate of the Taylor spectrum of  $S.$ 
\end{rem}

Similarly, the same is also true for the submodule $(M, T_z),$ where $T_z=(M_{z_1},\cdots,M_{z_n})$ can be viewed as the Toeplitz operator $n$-tuple with the coordinate functions.  We denote by $T_f$ the Toeplitz operator with the symbol function $f.$ 
\begin{cor}\label{ccor} Let $M$ be a submodule in $H_\lambda^2(\Omega),\lambda\in W_{\Omega,c}$ determined by an  analytic variety.   Suppose that $\phi: (\Omega,w)\rightarrow (\Omega,z) $ is a biholomorphism with $z=\phi(w)$, then  the following hold:\begin{enumerate}
\item  The $\mathbb{C}[w]$-module $(\phi^*(M),T_w)$ is $p$-essentially normal if and only if the  $\mathbb{C}[z]$-module  $(M,T_z)$  is $p$-essentially normal.
\item    The following two statements are equivalent.
\begin{enumerate}
\item $[T_{w_i},T_{w_j}^*]\in \mathcal{L}^p(M)$ for all $i,j=1,\cdots,n.$
\item $[T_{\phi_i},T_{\phi_j}^*]\in\mathcal{L}^p(M)$ for all $i,j=1,\cdots,n.$
\end{enumerate}
\end{enumerate}
\end{cor}


\section{The Taylor spectrum of the compression}

The previous sections have demonstrated the importance of the Taylor spectrum of the compression operator tuple in the investigation of the biholomorphic invariance of $p$-essential normality. Moreover, the calculation of the Taylor spectrum for bounded holomorphic multipliers is particularly an extension of the corresponding corona problem.  Thus in this section, we shall calculate the Taylor spectrum of the compression operator tuple of the analytic Hilbert quotient submodule on bounded symmetric domains. 

 Let  $T=(T_1,\cdots,T_n)$ be a commuting operator $n$-tuple on a Hilbert space $H.$ 
 Denote by $(T)$ the commutative  Banach subalgebra  generated by $T_1,\cdots,T_n$ and $\text{Id}$ in the Banach operator algebra $B(H).$  Let $\mathcal{B}$ be a  commutative  Banach algebra which contains $T_1,\cdots,T_n$ and $\text{Id}$. Then $T$ is called non-singular relative to $\mathcal{B}$ if the operator equation $$T_1U_1+\cdots+T_nU_n=\text{Id} $$ has a solution for $U_1,\cdots,U_n\in \mathcal{B}.$ The joint spectrum $Sp_{\mathcal{B}}(T)$ of $T$  relative to $\mathcal{B}$ is defined to be the set of all $w\in \mathbb{C}^n$  such that $T-w$ is singular  relative to $\mathcal{B}.$
  It is known that the Taylor spectrum $Sp(T)\subset Sp_{\mathcal{B}}(T)$ for every such Banach algebra $\mathcal{B};$ see \cite{Tay70A}. In particular, \begin{equation}\label{sspt}Sp(T)\subset Sp_{(T)}(T).\end{equation}

 Since  $(\Omega,z)$  is a bounded symmetric domain in $\mathbb{C}^n,$ it is the unit ball in the finite-dimensional normed linear space $(Z,\Vert \cdot\Vert),$ where $Z=\mathbb{C}^n$ is the Jordan triple and   $\Vert \cdot\Vert$ is the spectrum norm. Let  $(M_V^\bot,S_z)$ be a quotient submodule in  $H_\lambda^2(\Omega),\lambda\in W_{\Omega,c}$ determined by an  analytic variety $V.$  As we mentioned in Section 2, we can assume that there exist finitely many $f_1,\cdots,f_m\in  \mathcal{A}(\Omega)$ satisfying $$V=\{z\in\Omega:f_1(z)=\cdots=f_m(z)=0\}.$$ 
 We say that a submodule $M\subset H_\lambda^2(\Omega),\lambda\in W_{\Omega,c}$ is called  finitely generated if there exist finitely many $g_1,\cdots,g_m\in \mathcal{A}(\Omega)$ such that the closure $\text{Cl}[g_1\mathcal{O}(\bar{\Omega})+\cdots+g_m\mathcal{O}(\bar{\Omega})]=M.$ In this case we say that the submodule $M$ is finitely generated by $g_1,\cdots,g_m.$  The notion of finitely generated  $\mathcal{O}(\bar{\Omega})$-module used here can be roughly viewed as the replacement of coherent ideal sheaf of an analytic variety in a complex space in the global case. The main feature of the coherence is locally finitely generated. Moreover, the finitely generated  $\mathcal{O}(\bar{\Omega})$-module can be viewed as an analog of the (approximate) stable divisible $\mathbb{C}[z]$-module defined in \cite{Sha11}. 

 If $V=\{z\in\Omega:f_1(z)=\cdots=f_m(z)=0\}$ for some $f_1,\cdots,f_m\in \mathcal{O}(\bar{\Omega})$ and $ M_V$ is finitely generated by $g_1,\cdots,g_l\in \mathcal{O}(\bar{\Omega}),$ then $V=\{z\in\Omega:g_1(z)=\cdots=g_l(z)=0\},$ since $\mathcal{O}(\bar{\Omega})\subset H_\lambda^2(\Omega),\lambda\in W_{\Omega,c}.$ Thus in the following discussion, we only consider the case that $V=\{z\in\Omega:f_1(z)=\cdots=f_m(z)=0\}$ for some $f_1,\cdots,f_m\in \mathcal{O}(\bar{\Omega})$ and $ M_V$ is finitely generated by $f_1,\cdots,f_m.$

  The following proposition shows that  the Taylor spectrum of the   compression  $S_z$ on the  quotient submodule $M_V^\bot$
 is the closure of the  analytic variety $V$ if $M_V$ is finitely generated by some functions in $\mathcal{O}(\bar{\Omega}).$  Recall that $t\Omega=\{tz:z\in\Omega\}=\{z\in Z: \Vert z\Vert <t\} $ for every $t>0,$ which is biholomorphic to $\Omega=\{z\in Z: \Vert z\Vert <1\} .$   A complex linear functional on a complex commutative Banach algebra is said to be multiplicative if it is also a nonzero algebraic (ring) homomorphism.

 \begin{prop}\label{vspo} Suppose that $V=\{z\in\Omega:f_1(z)=\cdots=f_m(z)=0\}$ is an analytic variety with $f_1,\cdots,f_m\in \mathcal{O}(\bar{\Omega}).$ If $M_V$  is finitely generated by $f_1,\cdots,f_m,$  then the Taylor spectrum of $S_z$  is $$Sp(S_z)=\bar{V}.$$
 \end{prop}
 \begin{proof}  
Since $f_1,\cdots,f_m\in\mathcal{O}(\bar{\Omega}),$ it follows that there exists an open domain $D\supset \bar{\Omega}$ such that $f_1,\cdots,f_m\in\mathcal{O}(D).$ As the same proof in (\ref{exto}), we know that there exists a $t_0>1$ such that $$\Omega\subsetneq t\Omega\subsetneq t\bar{\Omega}\subsetneq D$$ for every $ 1< t\leq t_0.$ Thus $f_1,\cdots,f_m\in\mathcal{A}(t\Omega)$ for every $0<t\leq t_0.$ For every $0<t\leq t_0$, we define $V_t$ by $$V_t=\{z\in t\Omega:f_1(z)=\cdots=f_m(z)=0\}.$$ Clearly $V_1=V$ and  $\bar{V}_t=\{z\in t\bar{\Omega}:f_1(z)=\cdots=f_m(z)=0\}$ for every $ 1\leq t\leq t_0.$  We now fix an arbitrary $t\in (1,t_0]$ and let $w_t=(w_{t,1},\cdots,w_{t,n})\notin \bar{V}_t.$
 Denote $I_t$ by the closed ideal generated by $f_1,\cdots,f_m,z_1-w_{t,1},\cdots,z_n-w_{t,n}$ in the commutative Banach algebra $\mathcal{A}(t\Omega)$ where $1< t\leq t_0.$ We first prove that $I_t=\mathcal{A}(t\Omega).$  Otherwise, there exists a maximal ideal $J_t$ of $\mathcal{A}(t\Omega)$ satisfying $I_t\subset J_t,$ since every $\mathcal{A}(t\Omega)$ is a the commutative ring. The one-to-one correspondence between multiplicative linear functionals on $\mathcal{A}(t\Omega)$ and maximal  ideals in $\mathcal{A}(t\Omega)$ implies that there exists a unique multiplicative linear functional $\Phi_t$ satisfying $I_t\subset J_t=\text{Ker}\hspace{0.1em}\Phi_t.$ By the following Lemma \ref{mlfu}, we deduce that there exists a unique $w_t'\in t\bar{\Omega}$ such that $\Phi_t (f)=f(w_t')$ for every $f\in\mathcal{A}(t\Omega).$ Since $f_1,\cdots,f_m\in I_t\subset \text{Ker}\hspace{0.1em}\Phi_t,$ it follows that  $f_1(w_t')=\cdots=f_m(w_t')=0,$ namely $w_t'\in\bar{V}_t.$ On the other hand, the condition that $z_1-w_{t,1}\cdots,z_n-w_{t,n}\in I_t\subset \text{Ker}\hspace{0.1em}\Phi_t$ implies $w_t=w_t',$ which contradicts  the assumption $w_t\notin \bar{V}_t.$ This proves that $I_t=\mathcal{A}(t\Omega).$ Thus there exist $g_{t,1},\cdots,g_{t,m},h_{t,1},\cdots,h_{t,n}\in \mathcal{A}(t\Omega)\subset\mathcal{O}(\bar{\Omega})$ such that $$ f_1g_{t,1}+\cdots +f_mg_{t,m}+(z_1-w_{t,1})h_{t,1}+\cdots+(z_n-w_{t,n})h_{t,n}=1$$ on $t\Omega$ where $1<t\leq t_0.$ Combining this with Lemma \ref{fgo} and the fact $f_1,\cdots,f_m\in M_{V}$ follows that  \begin{equation}\label{siden}S_{z_1-w_{t,1}}S_{h_{t,1}}+\cdots+S_{z_n-w_{t,n}}S_{h_{t,n}}=\text{Id}\end{equation} on $M^\bot_V.$  On the other hand, as we proved in Lemma \ref{Hoo}, the $\mathcal{O}(\bar{\Omega})$-module action given by the multiplication of functions coincides with the $\mathcal{O}(\bar{\Omega})$-module action  (\ref{Tafh}) defined by the Taylor functional calculus. Combining this with the continuity of the Taylor functional calculus,   it implies that $S_f=P_{M^\bot_V}f(T_z)\in (S_z)$ for all $f\in \mathcal{O}(\bar{\Omega}).$  Therefore the identity (\ref{siden}) implies that $w_t\notin Sp_{(S_z)}(S_z),$ the joint spectrum of $S_z$  relative to the commutative Banach algebra $(S_z).$  Then the inclusion of spectrums (\ref{sspt}) and the arbitrariness of $t$ follows that $Sp(S_z)\subset \bar{V}_t$ for every $1<t\leq t_0.$ Together with the definition of $V_t$ yields $$Sp(S_z)\subset \bigcap_{1<t\leq t_0}\bar{V}_t=\bar{V}.$$

 It remains to prove the other direction, i.e., $Sp(S_z)\supset \bar{V}.$  Since the compactness of $Sp(S_z),$ it suffices to prove that $w\notin{V}$ if $w\notin Sp(S_z).$ Now suppose $w\notin Sp(S_z).$  By the definition of the Taylor spectrum, it follows that  $$\Lambda^{n-1} (M_{V}^\bot)\overset{D_{S_z-w}}{\longrightarrow}\Lambda^n (M_{V}^\bot)\rightarrow 0$$ is exact, and hence the boundary operator  $D_{S_z-w}$ is surjective on the above exterior product space. Since $\Lambda^n (M_{V}^\bot)=M_{V}^\bot\otimes (\eta_1\wedge\cdots\wedge\eta_n)\cong M_{V}^\bot$ up to the natural isomorphism,  we see that  $$M_{V}^\bot=S_{z_1-w_1}M_{V}^\bot+\cdots+S_{z_n-w_n}M_{V}^\bot.$$
 Therefore, for every $g\in H^2_\lambda(\Omega),\lambda\in W_{\Omega,c},$ there exist $g_1,\cdots,g_n\in M^\bot_V$ and $h\in M_V$ such that
 $$(z_1-w_1)g_1(z)+\cdots+(z_n-w_n)g_n(z)+h(z)=g(z).$$ If $w\in {\Omega},$ take $g(z)=\prod_{i=1}^n(z_i-w_i)+1,$ we have $h(w)=1\neq 0.$ Since $\text{Cl}[f_1\mathcal{O}(\bar{\Omega})+\cdots+f_m\mathcal{O}(\bar{\Omega})]=M_V,$ it follows that $f_1\mathcal{O}(\bar{\Omega})+\cdots+f_m\mathcal{O}(\bar{\Omega})$ is dense in $M_V,$  we  conclude that there exist $h_1,\cdots,h_m\in \mathcal{O}(\bar{\Omega})$ satisfying $$f_1(w)h_1(w)+\cdots+f_m(w)h_m(w)\neq 0,$$ which means $w\notin {V}.$ When $w\notin {\Omega},$ we see that $w\notin {V}$ is trivial. The proof is complete.
 \end{proof}

\begin{rem} Proposition \ref{vspo} can be viewed as a generalization of \cite[Theorem 5.1]{GuWk08} and \cite[Proposition 4.2]{Zha21}. The spectral result \cite[Theorem 2.5]{GuD06}  for quotient Beurling-type submodules on $\mathbb{B}^n$ with $n>1$ demonstrates that the analytic continuation assumption $f_1,\cdots,f_m\in \mathcal{O}(\bar{\Omega})$ in Proposition \ref{vspo} is necessary, since every inner function on the non-tube type domain $\mathbb{B}^n$ with $n>1$ is extremely oscillatory near every boundary point \cite[Lemma 2.4]{GuD06}. 

\end{rem}

 \begin{lem}\label{mlfu} Let $t>0,$ then every  complex multiplicative  linear functional on $\mathcal{A}(t\Omega)$  is an evaluation at some point of $t\bar{\Omega}.$
  \end{lem}
  \begin{proof} Clearly  $\mathcal{A}(t\Omega)$ is a commutative Banach algebra and the polynomial algebra $\mathcal{P}(Z)=\mathbb{C}[z]$ is its dense subset for every $t>0.$  We now fix a $t>0.$ Let $\Phi: \mathcal{A}(t\Omega) \rightarrow \mathbb{C} $ be a multiplicative  linear functional. Put $w=(\Phi(z_1),\cdots,\Phi(z_n)).$ We claim that $w\in t\bar{\Omega} .$ Otherwise, $w\notin t\bar{\Omega},$ that is $\Vert w\Vert >t>0.$  Define $\Vert \cdot\Vert_t$ by $\Vert z\Vert_t=\frac{1}{t}\Vert z\Vert$ for every $z\in Z.$ It is clear that $(Z,\Vert \cdot\Vert_t)$ is a finite-dimensional normed linear space (in fact a Banach space).  Moreover, $t\Omega$ is the open unit ball in $(Z,\Vert \cdot\Vert_t)$ and the norm of $w$ is $\Vert w\Vert_t=\frac{1}{t}\Vert w\Vert>1.$  By the Hahn-Banach theorem, there exists a  complex linear functional  $q$ on $(Z,\Vert \cdot\Vert_t)$ such that $q(w)=\Vert w \Vert_t>1$  and the operator norm $\Vert q\Vert=\sup_{z\in t\bar{\Omega}}\vert q(z)\vert=1.$ Since $Z$ is finite-dimensional, we see that every complex linear functional on $Z$ must be a homogeneous holomorphic polynomial of degree $1.$ Thus $q$ is a holomorphic polynomial such that its maximal modulus $\Vert q \Vert_{\infty}=\Vert q\Vert=1 $ on the unit ball $ t\Omega.$ Since  $\Phi: \mathcal{A}(t\Omega) \rightarrow \mathbb{C} $ is a multiplicative linear functional, it follows that the operator norm $\Vert \Phi\Vert =1$ and $\Phi(h)=h(w)$ for every polynomial $h\in\mathcal{P}(Z).$ In particular, $\Phi(q)=q(w)=\Vert w\Vert_t>1,$ which contradicts  the fact that $\vert \Phi(q)\vert\leq \Vert \Phi\Vert\cdot\Vert q\Vert_\infty=1.$ This proves that $w\in t\bar{\Omega}.$ Note that $\Phi(h)=h(w)$ for every  polynomial $h\in\mathcal{P}(Z).$ Combining the continuity of $\Phi$ with the density of $\mathcal{P}(Z)$ in $\mathcal{A}(t\Omega),$ it follows that $\Phi(f)=f(w)$ for every $f\in\mathcal{A}(t\Omega).$ This finishes the proof.
   \end{proof}

   Combined with the spectral mapping theorem, we have the following corollary.
  \begin{cor}\label{pTsp}  As assumed in Proposition \ref{vspo}, let $\phi=(\phi_1,\cdots,\phi_m)$ with $\phi_1,\cdots,\phi_m\in \mathcal{O}(\bar{\Omega}).$
  Denote $S_\phi=(S_{{\phi}_1},\cdots,S_{{\phi}_m}),$ then the Taylor spectrum of $S_\phi$ on $M_V^\bot$ is $$Sp(S_\phi)=\phi(\bar{V}).$$
   \end{cor}

As an application, we establish a result of the solvability of the corona problem for the quotient submodule.
Costea, Sawyer, and Wick established a corona theorem in \cite[Theorem 2]{CSW11} for the multiplier algebra  $\mathcal{M}(H^2_1(\mathbb{B}^n))$ of the whole Drury-Arveson space $H^2_1(\mathbb{B}^n).$
Recently,  Davidson,  Ramsey, and Shalit proposed a conjecture  in \cite[Remarks 5.5]{DRS15}:  

\begin{cnj}\label{cnj1}
The corona problem is solvable on the quotient Drury-Arveson submodule $M_V^\bot$ for its multiplier algebra $\mathcal{M}(M_V^\bot),$ where $V$ is the common zeros of some functions in  $H^2_1(\mathbb{B}^n).$   \end{cnj}

More precisely, the conjecture is that, for every $ f_1,\cdots,f_m\in\mathcal{M}(M_V^\bot)$ if there exists a $\delta>0$ such that $\vert f_1(z)\vert^2+\cdots +\vert f_m(z)\vert^2\geq\delta>0$ for every $z\in V,$ then for every $g\in M_V^\bot $ there exist $g_1,\cdots,g_m\in M_V^\bot$ satisfying $$f_1g_1+\cdots+f_mg_m=g.$$  
In particular, if the conjecture is true, then it implies that  $V$ is dense in the maximal ideal space of $\mathcal{M}(M_V^\bot).$ 
 We refer the reader to \cite{CSW11, DRS15}  for more details on the corona problem.
At the end of this section, we will consider this problem in the general case of bounded symmetric domains and give a supporting result. We note that  if $V=\Omega$ then the corona problem for  the  quotient  submodule $M_V^\bot$ is degenerated to the classical corona problem for  Hardy space $H^2_{\frac{n}{r}}(\Omega),$ Drury-Arveson space $H^2_{(r-1)\frac{a}{2}+1}(\Omega)$ and weighted Bergman space $H^2_{N+\gamma}(\Omega).$ As we mentioned in Section 2, the space $H^2_{\lambda}(\Omega)$ will not be  complete Nevanlinna-Pick for every $\lambda\in W_{\Omega}$ whenever its rank $r\geq2.$ Thus in the general case,  it is rather difficult to give a complete answer to the corona problem.

We first recall the definition of the multiplier algebra. Let $H$ be a reproducing kernel  Hilbert holomorphic function space on $\Omega$ with reproducing kernel $K,$  we denote by $\mathcal{M}(H)$ the multiplier algebra of $H,$ i.e., the set of functions $f\in\mathcal{O}(\Omega)$ satisfying $fg\in H$ for every $g\in  H.$ It follows from the closed graph theorem that if $f$ is a multiplier then multiplication operator $M_f$ with symbol $f$ is a bounded linear operator on $H.$ Thus $f\in\mathcal{O}(\Omega)$ is a multiplier if and only if $M_f$ is bounded. We identify the multiplier algebra with the bounded multiplication operator algebra. By direct calculations, we have  $M^\ast_fK_z=\overline{f(z)}K_z$ for every $z\in \Omega.$ Then the boundedness of $M_f$ implies that $f\in H^\infty(\Omega),$ the set of  bounded holomorphic functions. Indeed we obtain a criterion that a bounded holomorphic function $f$  is a multiplier if and only if $UK_z=\overline{f(z)}K_z$ can be extended to be a bounded linear operator on all of $H.$ In this case $U=M_f^\ast.$
Note that the constant function $\bm{1}\in H^2_\lambda(\Omega),\lambda\in W_{\Omega,c}.$ Then Lemma \ref{Hoo} follows that $$\mathcal{O}(\bar{\Omega})\subset\mathcal{M}(H^2_\lambda(\Omega))\subset H^2_\lambda(\Omega)\cap H^\infty(\Omega)$$ for $\lambda\in W_{\Omega,c}.$ Since $M_V^\bot =\overline{\text{span}}_{z\in V}\{K_{\lambda,z}\}
 \subset H^2_\lambda(\Omega)$ is a closed space, 
 it follows that $M_V^\bot$ is also a reproducing kernel  Hilbert holomorphic function space on $\Omega.$ The above criterion for the  multiplier shows that every $f\in \mathcal{M}(H^2_\lambda(\Omega))$ induces a  multiplier on $M_V^\bot,$ and we denote it by $f_V.$ Clearly, we see that if  $f,g \in \mathcal{M}(H^2_\lambda(\Omega)),$ then $f_V=g_V$ if and only if $f(z)=g(z)$ for all $z\in V,$ we define $\mathcal{M}(H^2_\lambda(\Omega))_V$ by $\mathcal{M}(H^2_\lambda(\Omega))_V=\{f_V: f\in \mathcal{M}(H^2_\lambda(\Omega)) \}.$ Similarly, we have $$\mathcal{O}(\bar{\Omega})_V\subset\mathcal{M}(H^2_\lambda(\Omega))_V\subset\mathcal{M}(M_V^\bot)\subset  H^\infty(\Omega),$$ where $\mathcal{O}(\bar{\Omega})_V=\{f_V:f\in \mathcal{O}(\bar{\Omega})\}$ is a  multiplier subalgebra of $\mathcal{M}(M_V^\bot) .$ When $\Omega=\mathbb{B}^n$ and $V$ is the common zeros of some functions in  Drury-Arveson space $H^2_1(\mathbb{B}^n),$ we know from \cite[Proposition 2.6]{DRS15} that $\mathcal{M}(H^2_1(\mathbb{B}^n))_V=\mathcal{M}(M_V^\bot)$ by the complete Nevanlinna-Pick property of $H^2_1(\mathbb{B}^n)$  and the restrictive condition for $V$ is necessary.

As assumed in Corollary \ref{pTsp}. Suppose $w=(w_1,\cdots,w_m)\notin \phi(\bar{V}),$ then for every $g\in  M^\bot_V\subset H^2_\lambda(\Omega),\lambda\in W_{\Omega,c},$ there exist $g_1,\cdots,g_m\in M^\bot_V$ such that $$(\phi_1-w_1)g_1+\cdots+(\phi_m-w_m)g_m =g.$$ In particular, if $\bm{0}\notin Sp(S_\phi)=\phi(\bar{V})=\overline{\phi(V)},$  equivalently  there exists a $\delta>0$ such that $\vert \phi_1(z)\vert^2+\cdots +\vert \phi_m(z)\vert^2\geq\delta>0$ for every $z\in V,$ which implies that for every $g\in  M^\bot_V,$ there exist $g_1,\cdots,g_m\in M^\bot_V$ such that $$\phi_1g_1+\cdots+\phi_m g_m =g.$$
This shows that the $M^\bot_V$-corona problem is solvable for multipliers belonging to $\mathcal{O}(\bar{\Omega})_V.$ 

\begin{prop}\label{corona1} Suppose that $V=\{z\in\Omega:f_1(z)=\cdots=f_m(z)=0\}$ is an analytic variety with $f_1,\cdots,f_m\in \mathcal{O}(\bar{\Omega}).$ If $M_V$  is finitely generated by $f_1,\cdots,f_m,$  then the corona problem of the quotient submodule $ M^\bot_V\subset H^2_\lambda(\Omega),\lambda\in W_{\Omega,c}$ is   solvable for the multiplier subalgebra $\mathcal{O}(\bar{\Omega})_V.$
 \end{prop}
Thus, in particular, we have the following corollary that provides supporting evidence for Conjecture \ref{cnj1}. 
\begin{cor}\label{corona2} Suppose that $V=\{z\in\mathbb{B}^n:f_1(z)=\cdots=f_m(z)=0\}$ is an analytic variety with $f_1,\cdots,f_m\in \mathcal{O}(\bar{\mathbb{B}}^n).$ If $M_V$  is finitely generated by $f_1,\cdots,f_m,$  then the corona problem of the quotient Drury-Arveson submodule  $ M^\bot_V$ is   solvable for the multiplier subalgebra $\mathcal{O}(\bar{\mathbb{B}}^n)_V.$
 \end{cor}

\begin{exm}\label{hycoro}
(1) Let $V=\{z_1z_2=0\} \cap\mathbb{B}^n$ which is  an algebraic hypersurface.
Then the corona problem of the quotient Drury-Arveson  submodule  $ M^\bot_V$ is   solvable for the multiplier subalgebra $\mathcal{O}(\bar{\mathbb{B}}^n)_V.$

(2) Suppose that $V\subset \mathbb{B}^n$ is a  homogeneous algebraic hypersurface that is smooth away from the origin. Then the corona problem of the quotient Bergman submodule  $ M^\bot_V$ is solvable for the multiplier subalgebra $\mathcal{O}(\bar{\mathbb{B}}^n)_V.$ An explicit example is the Fermat hypersurface $V=\{z_1^m+\cdots+z_n^m=0\}\cap\mathbb{B}^n$ for some $ m\geq 1.$
 \end{exm}
 \begin{proof} (1)  
  We first claim that the closure $\text{Cl}[z_1z_2\mathcal{O}(\bar{\mathbb{B}}^n)]=M_{V}.$  It suffices to prove that $\text{Cl}[z_1z_2\mathcal{O}(\bar{\mathbb{B}}^n)]\supset M_{V}.$ Suppose $f\in M_{V}\subset H^2_1(\mathbb{B}^n).$  By the Weierstrass division theorem, it implies that $f=z_1h_1+g_1$ in an open neighborhood of the origin, where $g_1$ is a holomorphic function with respect to the $n-1$ variables $z_2,\cdots,z_n.$ Since $f|_{V}=0,$ it follows that $g_1(z_2,\cdots,z_n)=0.$ Hence  \begin{equation}\label{fzh}f=z_1h_1\end{equation} in an open neighborhood of the origin.
  By the Weierstrass division theorem again,   $h_1=z_2h+g_2$ in an open neighborhood of the origin, where $g_2$ is a holomorphic function with respect to the $n-1$ variables $z_1,z_3,\cdots,z_n.$ The same reason follows that $g_2(z_1,z_3,\cdots,z_n)=0,$ namely $h_1=z_2h$ in an open neighborhood of the origin. Combining this with (\ref{fzh}) we know that  \begin{equation}\label{fzhh}f=z_1z_2h\end{equation} in an open neighborhood of the origin.

  Let $f(z)=\sum a_{l}z^{l}$ and  $h=\sum b_lz^l$ be their  Taylor expansions at the origin, where $l=(l_1,\cdots,\l_n)$ is the multi-index and $z^l=z_1^{l_1}\cdots z_n^{l_n},$ of course $l_1,\cdots,\l_n$ are nonnegative integers. It follows from (\ref{fzhh}) that $a_{1+l_1,1+l_2,\cdots,\l_n}=b_{l_1,l_2,\cdots,\l_n}$ for every  multi-index.
  Since $\mathbb{B}^n$ is circular, we see that the Taylor expansion $f$ at the origin is globally defined in the whole of $\mathbb{B}^n.$  Let $\frac{\partial^2}{\partial z_1\partial z_2}f=\sum c_lz^l$ be its Taylor expansion at the origin which is globally defined on the whole of $\mathbb{B}^n$ because the same reason. Clearly, we have $$\vert b_{l_1,\cdots,\l_n}\vert= \vert a_{1+l_1,1+l_2,\cdots,\l_n}\vert \leq (1+l_1 )(1+l_2)\vert a_{1+l_1,1+l_2,\cdots,\l_n}\vert=\vert c_l\vert $$ for every multi-index. Thus the power series $\sum b_lz^l$ is defined on the whole of $\mathbb{B}^n,$  and we denote it by $\tilde{h}.$ Clearly, $\tilde{h}=h$  in an open neighborhood of the origin,  combining this with (\ref{fzhh}) and the identity theorem of holomorphic functions, it follows that $f=z_1z_2\tilde{h}$ on $\mathbb{B}^n.$
  Denote $f_t(z)=f(tz)$  and $\tilde{h}_t(z)=\tilde{h}(tz)$ for $t>0.$ Clearly, if $0<t<1,$  then $$f_t=t^2z_1z_2\cdot \tilde{h}_t=z_1z_2\cdot t^2\tilde{h}_t\in z_1z_2\mathcal{O}(\bar{\mathbb{B}}^n) $$ and $f_t(z)=0$ for every $z\in V.$ Since $\mathcal{O}(\bar{\mathbb{B}}^n)\subset H^2_1(\mathbb{B}^n),$ it follows that $f_t\in M_{V}$ for every $0<t<1.$

 On the other hand, it follows from the norm formula that $$\Vert f_t\Vert_{H^2_1}^2=\sum_l \vert a_l\vert^2 t^{2\vert l\vert } \frac{l!}{\vert l\vert !}\leq \sum_l \vert a_l\vert^2 \frac{l!}{\vert l\vert !}=\Vert f\Vert_{H^2_1}^2,$$ where $l!=l_1!\cdots l_n!$ and $\vert l\vert=l_1+\cdots+\l_n.$ Hence $g_m=f_{1-\frac{1}{m+1}}$   converges pointwise to $f$ and is a bounded sequence in the reproducing kernel Hilbert space $M_{V}.$  By the  Banach-Alaoglu theorem, the sequence $\{g_m\}$ possesses a subsequence that converges weakly. Since  $g_m$   converges to $f$ pointwise, it follows that the weak limit must coincide with  $f.$ Thus $f$ belongs to the weak closure of $B_f,$ where  $$B_f=\{g\in z_1z_2\mathcal{O}(\bar{\mathbb{B}}^n):  \Vert g\Vert_{H^2_1}\leq \Vert f\Vert_{H^2_1}\}\subset M_{V}.$$ Note that  $B_f$ is convex, it follows from the Hahn-Banach theorem that the weak closure of $B_f$ coincides with the norm closure of  $B_f.$  Thus there exists a sequence $\{F_m\}\subset \mathcal{O}(\bar{\mathbb{B}}^n)$ such that $z_1z_2F_m$ converges to $f$ in the norm topology of $ H^2_1(\mathbb{B}^n),$ namely $f\in \text{Cl}[z_1z_2\mathcal{O}(\bar{\mathbb{B}}^n)].$  Thus $\text{Cl}[z_1z_2\mathcal{O}(\bar{\mathbb{B}}^n)]\supset M_{V},$ which proves the claim. Then  Corollary \ref{corona2} follows the desired result.

 (2) There exists a homogeneous polynomial  $h$ such that $V=\{h(z)=0\}\cap \mathbb{B}^n.$   It follows from \cite[Theorem A.3]{DTY16} that $\text{Cl}[h\mathcal{O}(\bar{\mathbb{B}}^n)]=M_V.$
 Combined with Corollary \ref{corona2}, we know that the corona problem of the quotient Bergman submodule  $ M^\bot_V$ is solvable for the multiplier subalgebra $\mathcal{O}(\bar{\mathbb{B}}^n)_V.$
 \end{proof}

 When $V\subset \mathbb{B}^n$ is a  homogeneous algebraic variety and $I_V$ the corresponding polynomial radical ideal, then the condition $M_V=\text{Cl}[I_V]$ implies the finitely generated property of $M_V.$ In general, it is nontrivial to verify whether condition $M_V=\text{Cl}[I_V]$ holds. The method used in  \cite{DTY16} for the case of Bergman space can not be directly applied to the Drury-Arveson space since the Drury-Arveson norm can not be uniformly controlled by the uniform norm. The method we used here can actually at least handle the case of the  Drury-Arveson submodule  $M_V$ for which the algebraic hypersurface $V$ is determined by some monomial. 

\section{applications}
  In this end section,  we pay attention to the case of $\Omega=\mathbb{B}^n,n>1$ and give two direct applications of the biholomorphic invariance of $p$-essential normality. 
  \subsection{Geometric Arveson-Douglas conjecture} The Geometric Arveson-Douglas  conjecture  is a variation of the original Arveson-Douglas conjecture and plays a dominant role in the development of the $p$-essential normality theory.  We refer the reader to the survey \cite{GuW20}.  From Theorem 2 or Theorem  \ref{hwaa}, we know that the $p$-essential normality of quotient Hardy and Bergman submodules determined by the affine  (irreducible) homogeneous
  algebraic variety in $\mathbb{B}^n$ is biholomorphic invariant, thus the Geometric Arveson-Douglas conjecture is equivalent to the following Conjecture \ref{cnj1}. Let $\emptyset \neq V\subset \mathbb{B}^n$ be an affine homogeneous algebraic variety,  which means that there exist finitely many holomorphic homogeneous polynomials $q_1,\cdots,q_m$ satisfying $V=\{z\in\mathbb{B}^n: q_1(z)=\cdots=q_m(z)=0\}.$ Define $(M_V, T_z)$ to be the submodule with the coordinate multipliers in the Bergman module (or Hardy module), its quotient submodule with the compression of the coordinate multipliers  $T_z$ is denoted by $(M^\bot_V, S_z).$ Let $S_\phi=(S_{{\phi}_1},\cdots, S_{{\phi}_n}),$ for every $\phi=(\phi_1,\cdots,\phi_n)\in \text{Aut}(\mathbb{B}^n).$

\begin{cnj}\label{cnj1} Let $V\subset \mathbb{B}^n$ be an affine homogeneous algebraic variety. Then for some (every) $\phi\in \text{Aut}(\mathbb{B}^n),$ the quotient submodule $(M^\bot_V,S_\phi)$ is $p$-essentially normal for all $p>\text{dim}_\mathbb{C}\hspace{0.1em}V.$
\end{cnj}

Moreover, if the Taylor spectrum  $Sp(S_z)\subset \bar{\mathbb{B}}^n,$ then by Remark \ref{szfg} the Geometric Arveson-Douglas conjecture is equivalent to the following conjecture.
\begin{cnj}\label{cnj2} Let $V\subset \mathbb{B}^n$ be an affine homogeneous algebraic variety.
If  the Taylor spectrum  $Sp(S_z)\subset \bar{\mathbb{B}}^n,$
then
 $$[S_{f},S_{g}^*]\in\mathcal{L}^p(M_V^\bot),\quad \forall f,g\in \mathcal{O}(\bar{\mathbb{B}}^n),$$  for all $p>\text{dim}_\mathbb{C}\hspace{0.1em}V.$
\end{cnj}
\subsection{ Hyperrigidity and $\infty$-essential normality}  
An operator system $\mathcal{S}$ is said to be hyperrigid if for every non-degenerate representation $\pi :C^\ast(\mathcal{S})\rightarrow B(H)$ on a separable complex Hilbert space $H,$ then $\pi$ is the unique unital completely positive extension of the restriction $\pi|_{\mathcal{S}},$ where $C^\ast(\mathcal{S})$ is the $C^\ast$-algebra generated by $\mathcal{S}.$    An operator tuple $S$ on a Hilbert space $H$ is said to be hyperrigid if  the the operator system $\mathcal{S}$ genereted by $S$ is hyperrigid. The notion of hyperrigidity for an operator system was first proposed by Arveson to study the boundary representation. The hyperrigidity conjecture for a separable operator system proposed by Arveson has drawn increased attention in recent years; see \cite{CT21,KS16} and references therein.

  In general, it is typically difficult to determine whether an operator system is hyperrigid, since it requires checking that every non-degenerate representation of the corresponding $C^\ast$-algebra has the unique extension property.
Recently,  Kennedy and Shalit \cite[Theorem 4.12]{KS16} established a deep connection between  $\infty$-essential normality and hyperrigidity,  i.e.,  the quotient Drury-Arveson submodule $(\mathcal{Q}_I,S_z)$ determined by a sufficiently nontrivial homogenous polynomial ideal  $I\subset\mathbb{C}[z]$ is  $\infty$-essentially normal if and only if the compression $n$-tuple $S_z$ is hyperrigid. 
They also confirmed in \cite[Corollary 5.3]{KS16}  that the same is also true on  the Hilbert module $H_\lambda^2(\mathbb{B}^n), 1\leq\lambda<d,$ where $H_\lambda^2(\mathbb{B}^n)$ is defined in Section 2.  More recently, Clou\^{a}tre and Timko  \cite[Theorem 1.2]{CT21}  extended this result to the analytic Hilbert modules with a maximal regular unitarily invariant complete Nevanlinna-Pick kernels on $\mathbb{B}^n.$ These results provide an approach to characterizing hyperrigidity by essential normality.  We know from Theorem 1 or Theorem \ref{hwaa} that 
$(\mathcal{Q}_I,S_\phi)$  is  $\infty$-essentially normal for some (arbitrary) $\phi\in \text{Aut}(\mathbb{B}^n),$ and hence it is intuitive that  $S_\phi$ should be hyperrigid for every $\phi\in \text{Aut}(\mathbb{B}^n).$  We formulate it in the following result, which extends the above Kennedy-Shalit theorem to a broader range of compression tuples and quotient submodules. 

\begin{thm}\label{snhh} Let $I$ be a sufficiently nontrivial homogenous polynomial ideal of $\mathbb{C}[z].$ 
For the quotient  submodule  $(\mathcal{Q}_I,S_z)$ in $H_\lambda^2(\mathbb{B}^n), 1\leq\lambda<d,$ 
then  the following are equivalent:
\begin{enumerate}
\item $(\mathcal{Q}_I,S_z)$  is $\infty$-essentially normal.
\item  $(\mathcal{Q}_I,S_\phi)$ is $\infty$-essentially normal for every (or some) $\phi\in \text{Aut}(\mathbb{B}^n).$
\item $[S_f,S_g^\ast]\in \mathcal{L}^\infty(\mathcal{Q}_I )$ for all $f,g\in \mathcal{O}( \bar{\mathbb{B}}^n).$
\item $[S_f,S_f^\ast]\in \mathcal{L}^\infty(\mathcal{Q}_I )$ for all $f\in \mathcal{O}( \bar{\mathbb{B}}^n).$
\item  $S_z$ is hyperrigid.
\item   $S_\phi$ is hyperrigid for every (or some) $\phi\in \text{Aut}(\mathbb{B}^n).$
\end{enumerate}
\end{thm}
\begin{proof} (1)$\iff$(2): For the quotient submodule $(\mathcal{Q}_I,S_z),$ we can show  the Taylar spectrum $Sp(S_z)=Z(I)\cap\bar{\mathbb{B}}^d,$ by the the similar method used in proof of Proposition \ref{vspo}. Lemma \ref{Hoo} implies that $\phi(S_z)=S_\phi$ for every $\phi\in \text{Aut}(\mathbb{B}^n).$ Then Theorem 1 (taking $c=1,\bm{d=0}$) follows the equivalence.

(1)$\iff$(3): It suffices to prove the  implication (1)$\implies$(3). Since $Sp(S_z)\subset\bar{\mathbb{B}}^d,$ the same argument of the Taylor functional calculus as we used in the proof of Theorem \ref{th1} (taking $c=1,\bm{d=0}$) implies that (3) holds.

(3)$\iff$(4):  It comes from the Fuglede-Putnam theorem.

(1)$\iff$(5): This is \cite[Corollary 5.3]{KS16}.

(5)$\iff$(6) Suppose $\phi\in \text{Aut}(\mathbb{B}^n).$ 
Let $$\mathcal{S}_\phi=\text{span}\hspace{0.1em}\{\text{Id},S_{\phi_1},\cdots,S_{\phi_n},S_{\phi_1}^\ast,\cdots,S_{\phi_n}^\ast\}$$ 
be the operator system generated by the operator $n$-tuple $S_\phi$  for every $\phi\in \text{Aut}(\mathbb{B}^n).$ Denote $\overline{\text{A1g}}(\mathcal{S}_\phi)$ by the clsosure in the operator norm. 
We first claim that 
\begin{equation}\label{Alca} \overline{\text{A1g}}(\mathcal{S}_\phi)=\overline{\text{A1g}}(\mathcal{S}_z),\hspace{0.3em} C^\ast(\mathcal{S}_\phi)=C^\ast(\mathcal{S}_z) , \quad \forall \phi\in\text{Aut}(\mathbb{B}^n) . \end{equation}
It follows from Lemma \ref{bihg} that  $\phi=(\phi_1,\cdots,\phi_n)\in \mathcal{O}( \bar{\mathbb{B}}^n)\otimes \mathbb{C}^n,$ which means that all $\phi_i$ are holomorphic on an open ball $t\mathbb{B}^n$ with radius $t>1,1\leq i\leq n.$  Since $t\mathbb{B}^n$ are circular domain,  all $\phi_i$ can be uniformly approximated by polynomials
on $ \bar{\mathbb{B}}^n, 1\leq i\leq n.$ And hence $\mathcal{S}_\phi\subset \overline{\text{A1g}}(\mathcal{S}_z) \subset C^\ast(\mathcal{S}_z)$ by the continuity and uniqueness of the Taylor functional calculus. It follows that 
$$ \overline{\text{A1g}}(\mathcal{S}_\phi)\subset\overline{\text{A1g}}(\mathcal{S}_z),\hspace{0.3em} C^\ast(\mathcal{S}_\phi)\subset C^\ast(\mathcal{S}_z).$$
Note that  $Sp(S_z)=Z(I)\cap\bar{\mathbb{B}}^d,$  it follows from Lemma \ref{Hoo}  that $\phi(S_z)=S_\phi.$ Similarly, we can deduce that $\mathcal{S}_z=\mathcal{S}_{\phi^{-1}\circ\phi}\subset \overline{\text{A1g}}(\mathcal{S}_\phi) \subset C^\ast(\mathcal{S}_\phi).$
This implies that $$ \overline{\text{A1g}}(\mathcal{S}_z)\subset\overline{\text{A1g}}(\mathcal{S}_\phi),\hspace{0.3em} C^\ast(\mathcal{S}_z)\subset C^\ast(\mathcal{S}_\phi).$$
Then claim (\ref{Alca}) is proved. 

Now suppose  $S_z$ is hyperrigid, we prove that $S_\phi$ is hyperrigid.  
Let $\pi:C^\ast(\mathcal{S}_\phi)\rightarrow B(H)$ be a  non-degenerate representation on a Hilbert space $H.$    It follows from the identity  $C^\ast(\mathcal{S}_\phi)=C^\ast(\mathcal{S}_z)$ and the hyperrigidity of $S_z$ that there is a  unique unital completely positive extension of the restriction $\pi|_{\mathcal{S}_z}.$ Then it is clear that $\pi$ is a unital completely positive extension of the restriction $\pi|_{\mathcal{S}_\phi}.$ If $\pi'$ is another unital completely positive extension of the restriction $\pi|_{\mathcal{S}_\phi}.$ Since  $\overline{\text{A1g}}(\mathcal{S}_\phi)=\overline{\text{A1g}}(\mathcal{S}_z)$ and the the continuity of $\pi,\pi',$ it follows that $$\pi'|_{\overline{\text{A1g}}(\mathcal{S}_\phi)}=\pi|_{\overline{\text{A1g}}(\mathcal{S}_\phi)}=\pi|_{\overline{\text{A1g}}(\mathcal{S}_z)}.$$ Thus $\pi'$ is a unital completely positive extension of the restriction $\pi|_{\mathcal{S}_z}.$ The uniqueness implies that $\pi'=\pi.$ Thus ${S}_\phi$ is hyperrigid. This shows the implication (5)$\implies$(6).
Similarly, we can show that ${S}_z$ is hyperrigid if ${S}_\phi$ is hyperrigid.
\end{proof}
 We see from Theorem \ref{snhh} that operator systems $\mathcal{S_\phi},\phi\in \text{Aut}(\mathbb{B}^n)$ support Arveson's hyperrigidity conjecture for every, provided  $(\mathcal{Q}_I,S_z)$ satisfies the Arveson-Douglas (or Arveson) conjecture.

  \vspace{0.5cm}

{\noindent{\bf{Acknowledgements.}} 

The author would like to thank Prof. Kunyu Guo and  Prof. Harald Upmeier for their helpful discussions.  
The author was partially supported by the National Natural Science Foundation of China (12201571).

 \bibliographystyle{plain}

\begin{thebibliography}{99}
\small

\bibitem{AZ03}
J. Arazy, G. Zhang, \emph{Homogeneous multiplication operators on bounded symmetric domains,} J. Funct. Anal. \textbf{202} (2003), no. 1, 44-66.
\bibitem{Arv98}
W. Arveson, \emph{Subalgebras of $C^*$-algebras $\uppercase\expandafter{\romannumeral3}$: Multivariable operator theory,} Acta Math. \textbf{181} (1998), 159-228.
\bibitem{Arv00}W. Arveson,\emph{The curvature invariant of a Hilbert module over $\mathbb{C}[z_1,\cdots,z_d]$,} J. Reine Angew. Math. \textbf{522} (2000),  173-236.
\bibitem{Arv05}
W. Arveson, \emph{$p$-summable commutators in dimension $d$,} J. Oper. Theory \textbf{54} (2005), 101-117.
\bibitem{Arv07}
W. Arveson, \emph{Quotients of standard Hilbert modules,} Trans. Amer. Math. Soc. \textbf{359} (2007), no.12, 6027-6055.

 \bibitem{Chu21}
 C. Chu, \emph{Bounded symmetric domains in Banach spaces}, World Scientific, Hackensack, NJ, 2021.
  
  \bibitem{CT21} R. Clou\^{a}tre, E. Timko, \emph{Gelfand transforms and boundary representations of complete
              {N}evanlinna-{P}ick quotients}, Trans. Amer. Math. Soc. \textbf{374} (2021), 2107-2147.
              
\bibitem{Con85} A. Connes, \emph{Non-commutative differential geometry,} Inst. Hautes \' Etudes Sci. Publ. Math. (1985), no. 62, 257-360. 

  \bibitem{CSW11}
 S. Costea, E. Sawyer and B. Wick, \emph{The Corona Theorem for the Drury-Arveson Hardy space and other holomorphic Besov-Sobolev spaces on the unit ball in $C^n$}, Anal. PDE \textbf{4} (2011), 499-550.
\bibitem{DRS11}
K. Davidson, C. Ramsey, and O. Shalit, \emph{The isomorphism problem for some universal operator algebras,}
Adv. Math. \textbf{228} (2011),  167-218.
\bibitem{DRS15}
K. Davidson, C. Ramsey and O. Shalit, \emph{Operator algebras for analytic varieties,} Trans. Amer. Math. Soc. \textbf{367} (2015), 1121-1150.
\bibitem{Din22}
L. Ding, \emph{Schatten class Bergman-type and Szeg\"o-type operators on bounded symmetric domains,} Adv. Math. \textbf{401} (2022),  No. 108314, 27 pp.
\bibitem{DTY16}
R. Douglas,  X. Tang and G. Yu, \emph{An analytic Grothendieck Riemann Roch theorem,} Adv. Math. \textbf{294} (2016), 307-331.
\bibitem{DoWy17}
R. Douglas, Y. Wang, \emph{Geometric Arveson-Douglas conjecture and holomorphic extensions,} Indiana Univ. Math. J. \textbf{66} (2017), 1499-1535.
\bibitem{EnE15}
M. Engli\v{s}, J. Eschmeier, \emph{Geometric Arveson-Douglas conjecture,} Adv. Math. \textbf{274} (2015), 606-630.
\bibitem{GF76}
H. Grauert,  K. Fritzsche, \emph{Several complex variables,} Graduate Texts in Mathematics, Vol. 38. Springer-Verlag, New York-Heidelberg, 1976.
\bibitem{GuD06}
K. Guo, Y. Duan, \emph{Spectral properties of quotients of Beurling-type submodules of the Hardy module over the unit ball,} Studia Math. \textbf{177} (2006), 141-152.
\bibitem{GuWk08}
K. Guo, K. Wang, \emph{Essentially normal Hilbert modules and $K$-homology,} Math. Ann. \textbf{340} (2008), 907-934.
\bibitem{GWZ23}
K. Guo, P. Wang and C. Zhao, \emph{Essentially normal quotient weighted Bergman modules over the bidisk and distinguished varieties,}  Adv. Math. \textbf{432} (2023), no. 109266,  30 pp.
\bibitem{GuW20}
K. Guo, Y. Wang, \emph{A survey on the Arveson-Douglas conjecture,}   Oper. Theory Adv. Appl. \textbf{278} (2020), Birkh\"auser/Springer, Cham, 289-311.
\bibitem{KS12} 
M. Kennedy, O. Shalit, \emph{Essential normality and the decomposability of algebraic varieties,}
New York J. Math. \textbf{18} (2012), 877-890.
\bibitem{KS16} 
M. Kennedy, O. Shalit, \emph{Essential normality, essential norms and hyperrigidity,} J. Funct. Anal. \textbf{268} (2015) 2990-3016; Corrigendum: J. Funct. Anal. \textbf{270} (2016) 2812-2815.
\bibitem{Loo77}
O. Loos, \emph{Bounded symmetric domains and Jordan Pairs,} Mathematical Lectures, The University of California,  1977.

\bibitem{Tay70A}
J. Taylor, \emph{The analytic-functional calculus for several commuting operators,} Acta Math. \textbf{125} (1970), 1-38.
\bibitem{Put82}
M. Putinar, \emph{The superposition property for Taylor's functional calculus,} J. Operator Theory \textbf{7} (1982), 149-155.
\bibitem{Sal89}
N. Salinas,  \emph{The {$\overline\partial$}-formalism and the $C^*$-algebra of the Bergman $n$-tuple}, J. Operator Theory \textbf{22} (1989), 325-343.
 \bibitem{Sha11}
O. Shalit, \emph{Stable polynomial division and essential normality of graded Hilbert modules,} J.
Lond. Math. Soc. \textbf{83} (2011), 273-289.
 \bibitem{Upm84} H. Upmeier, \emph{Toeplitz $C^\ast$-algebras on bounded symmetric domains,} Ann. of Math. (2) \textbf{119}  (1984),  549-576.
\bibitem{Upm96}
 H. Upmeier, \emph{Toeplitz operators and index theory in several complex variables.} Operator
Theory: Advances and Applications \textbf{81}, Birkh\"auser Verlag, Basel, 1996.
\bibitem{UpW16}
 H. Upmeier, K. Wang, \emph{Dixmier trace for Toeplitz operators on symmetric domains,} J. Funct. Anal. \textbf{271} (2016), 532-565.
 \bibitem{Zha21}S. Zhang, \emph{Essential normality for Beurling-type quotient modules over tube-type domains.} Integral Equations Operator Theory \textbf{93} (2021),  No. 3, 18 pp. 
\end{thebibliography}
 
\end{document}